\documentclass[12pt]{amsart}

\usepackage{amsthm,amsfonts,amsmath,verbatim,mathrsfs,amssymb,verbatim,cite,fouridx,youngtab}
\usepackage[left=2.5cm, right=2.5cm, top=3.5cm, bottom=3.5cm]{geometry}
\usepackage[usenames]{color}
\usepackage{todonotes}
\usepackage{hyperref}
\usepackage{breakurl}
\usepackage{url}

\usetikzlibrary{math}

\makeatletter
\newenvironment{subtheorem}[1]{%
  \def\subtheoremcounter{#1}%
  \refstepcounter{#1}%
  \protected@edef\theparentnumber{\csname the#1\endcsname}%
  \setcounter{parentnumber}{\value{#1}}%
  \setcounter{#1}{0}%
  \expandafter\def\csname the#1\endcsname{\theparentnumber.\alph{#1}}%
  \ignorespaces
}{%
  \setcounter{\subtheoremcounter}{\value{parentnumber}}%
  \ignorespacesafterend
}
\makeatother
\newcounter{parentnumber}


\newcounter{parentnumberr}

\newtheorem{theorem}{Theorem}[section]
\newtheorem{lemma}[theorem]{Lemma}
\newtheorem{proposition}[theorem]{Proposition}

\newtheorem{conjecture}[theorem]{Conjecture}
\newtheorem{problem}[theorem]{Problem}
\theoremstyle{definition}

\numberwithin{equation}{section}

\DeclareMathOperator{\arm}{arm}
\DeclareMathOperator{\leg}{leg}
\DeclareMathOperator{\hook}{hook}

\DeclareMathOperator{\eup}{e}
\DeclareMathOperator{\iup}{i}

\DeclareMathOperator{\BF}{BF}

\DeclareMathOperator{\BFset}{bf}

\newcommand{\blue}[1]{{\color{blue}  #1}}
\newcommand{\red}[1]{{\color{red}  #1}}
\newcommand{\magenta}[1]{{\color{magenta}  #1}}
\newcommand{\green}[1]{{\color{green}  #1}}

\renewcommand{\subseteq}{\subset}

\newcommand{\la}{\lambda}

\newcommand{\floor}[1]{\lfloor#1\rfloor}
\newcommand{\HH}{\mathscr{H}}
\newcommand{\HHbeen}{\HH^{\textup{(b)}}}
\newcommand{\HHreen}{\HH^{\textup{(r)}}}
\newcommand{\HHb}[1]{\HH_{#1}^{\textup{(b)}}}
\newcommand{\HHr}[1]{\HH_{#1}^{\textup{(r)}}}
\newcommand{\PP}{\mathscr{P}}
\newcommand{\KK}{\mathscr{K}}

\newcommand{\CC}{\mathscr{C}}

\renewcommand{\leq}{\leqslant}
\renewcommand{\geq}{\geqslant}

\newcommand{\abs}[1]{\lvert#1\rvert}

\newcommand{\core}[1]{\textup{$r$-core}(#1)}

\newcommand{\Mod}[1]{\:(\textup{mod}\:#1)}

\newcommand{\Z}{\mathbb Z}

\newcommand{\Complex}{\mathbb C}

\begin{document}

\title{Modular Nekrasov--Okounkov formulas}

\author{Adam Walsh and S. Ole Warnaar}

\address{School of Mathematics and Statistics,
The University of Melbourne, VIC 3010, Australia}
\email{adamwalsh.central@gmail.com}

\address{School of Mathematics and Physics,
The University of Queensland, Brisbane, QLD 4072,
Australia}
\email{o.warnaar@maths.uq.edu.au}
\thanks{Work supported by the Australian Research Council}


\begin{abstract}
Using Littlewood's map, which decomposes a partition into its $r$-core
and $r$-quotient, Han and Ji have shown that many well-known hook-length
formulas admit modular analogues.
In this paper we present a variant of the Han--Ji `multiplication theorem'
based on a new analogue of Littlewood's decomposition.
We discuss several applications to hook-length formulas, one of which 
leads us to conjecture a modular analogue of the $q,t$-Nekrasov--Okounkov
formula.
\end{abstract}

\dedicatory{To Christian Krattenthaler on the occasion of his 60th birthday}

\maketitle

\section{Introduction}

Hook-length formulas abound in combinatorics and representation theory.
Perhaps the most famous example is the formula for $f^{\la}$, the number
of standard Young tableaux of shape $\la$, which was discovered in 1954
by Frame, Robinson and Thrall \cite{FRT54}.  
If $\la\vdash n$ and $\HH(\la)$ denotes the multiset of hook-lengths of 
the partition $\la$ (we refer to Section~\ref{Sec_Partitions} for notation 
and definitions), then
\begin{equation}\label{Eq_FRT}
f^{\la}=\frac{n!}{\prod_{h\in\HH(\la)} h}.
\end{equation}

A much more recent identity in the spirit of \eqref{Eq_FRT} is the
Nekrasov--Okounkov formula.
It was discovered independently by Nekrasov and Okounkov \cite{NO06} in
their work on random partitions and Seiberg--Witten theory, and by
Westbury \cite{Westbury06} in his work on universal characters for
$\mathfrak{sl}_n$.
The form in which the formula is commonly stated is that of Nekrasov and
Okounkov (see \cite[Equation (6.12)]{NO06})
\begin{equation}\label{Eq_NO}
\sum_{\la\in\PP} T^{\abs{\la}} \prod_{h\in\HH(\la)}
\Big(1-\frac{z}{h^2}\Big)
=\prod_{k\geq 1}(1-T^k)^{z-1}
\end{equation}
rather than Westbury's hook-length formula for the D'Arcais polynomials 
$P_n(z)$, defined by the expansion \cite{Darcais13}
\[
\prod_{k\geq 1}\frac{1}{(1-T^k)^z}=
\sum_{n=0}^{\infty} P_n(z) T^n,
\]
and implied by Propositions 6.1 \& 6.2 of his paper \cite{Westbury06}.

The Nekrasov--Okounkov formula has attracted significant attention
in a number of different areas of mathematics and physics, including
algebraic geometry, combinatorics, number theory and string theory. 
It has seen $q$-generalisations \cite{DH11,INRS12}, $q,t$-generalisations 
\cite{AK09,AK13,CDDP15,CRV16,IKS10,HLRV11,HRV08,PS09,RW18},
an elliptic analogue \cite{LLZ06,RW18,Waelder08}, modular analogues
\cite{DH11,Han10,HJ11}, and generalisations to the affine Lie algebras 
$\mathrm{C}_n^{(1)}$ \cite{Petreolle16} and $\mathrm{D}_{n+1}^{(2)}$ 
\cite{Petreolle16b}.
It has also sparked the study of several combinatorial problems on
partitions and hook-length statistics, see e.g.,
\cite{AK18,Amdeberhan12,DHX17,FKMO08,HN18,HX18,HX16,Panova12,Stanley10},
and has given new impetus to the study of the arithmetical properties
of Euler-type products, see e.g., \cite{CKW09,CMSZ17,GLV18,HO11}.

\medskip

Let 
\[
(a_1,a_2,\dots,a_k;q_1,q_2,\dots,q_m)_{\infty}
:=\prod_{i=1}^k \prod_{j_1,\dots,j_m\geq 0}
\big(1-a_i q_1^{j_1}q_2^{j_2}\cdots q_m^{j_m}\big)
\]
be a multiple $q$-shifted factorial and, for $\la$ a partition and $r$
a positive integer, let $\HH_r(\la)$ denote the multiset of hook-lengths 
of $\la$ that are congruent to $0$ modulo $r$. 
Then one particularly interesting generalisation of the
Nekrasov--Okounkov formula is Han's modular analogue
\cite[Theorem 1.3]{Han10}
\begin{equation}\label{Eq_NOmod}
\sum_{\la\in\PP} T^{\abs{\la}}S^{\abs{\HH_r(\la)}}
\prod_{h\in\HH_r(\la)}
\Big(1-\frac{z}{h^2}\Big)
=\frac{(T^r;T^r)_{\infty}^r}
{(T;T)_{\infty}(ST^r;ST^r)_{\infty}^{r-z/r}}.
\end{equation}
He proved this identity by combining \eqref{Eq_NO} with Littlewood's
decomposition, which is a generalisation of Euclidean division to
integer partitions that has played a key role in the modular
representation theory of the symmetric group.
The question of modular analogues of hook-length formulas was further
pursued by Han in subsequent papers with Dehaye \cite{DH11} and 
Ji \cite{HJ11}.
The most general statement was formulated in this last paper.

\begin{theorem}[{`Multiplication theorem' \cite[Theorem 1.5]{HJ11}}]
\label{Thm_HanJi}
For $r$ a positive integer and $\rho$ a function on the positive integers,
let $f_r(T)$ be the formal power series defined by
\begin{equation}\label{Eq_seed}
f_r(T):=\sum_{\la\in\PP}T^{\abs{\la}}\prod_{h\in \HH(\la)}\rho(rh).
\end{equation}
Then
\begin{equation}\label{Eq_trafo}
\sum_{\la\in\PP}T^{\abs{\la}}S^{\abs{\HH_r(\la)}}
\prod_{h\in \HH_r(\la)} \rho(h)
=\frac{(T^r;T^r)_{\infty}^r}{(T;T)_{\infty}}\,\big(f_r(ST^r)\big)^r.
\end{equation}
\end{theorem}
Whenever $\rho$ is chosen such that $f_r(T)$ admits a closed-form
expression, the above theorem immediately implies a modular analogue of
the hook-length formula
\[
\sum_{\la\in\PP}T^{\abs{\la}}\prod_{h\in \HH(\la)}\rho(h)=f_1(T).
\]

In this paper we describe a variant of Littlewood's decomposition which
implies an analogue of the Han--Ji multiplication theorem for
hook-length formulas involving the hook-lengths of squares
of partitions that have trivial leg-length. 
Combining this with the Han--Ji multiplication theorem,
suggests a modular analogue of the $q,t$-analogue
of \eqref{Eq_NO}, see e.g., \cite[Theorem 1.0.2]{CRV16} or
\cite[Theorem 1.3]{RW18}
\begin{equation}\label{Eq_qtNO}
\sum_{\la\in\PP} T^{\abs{\la}} \prod_{s\in\la}
\frac{(1-uq^{a(s)+1}t^{l(s)})(1-u^{-1}q^{a(s)}t^{l(s)+1})}
{(1-q^{a(s)+1}t^{l(s)})(1-q^{a(s)}t^{l(s)+1})}
=\frac{(uqT,u^{-1}tT;q,t,T)_{\infty}}
{(T,tT;q,t,T)_{\infty}},
\end{equation}
where $a(s)$ and $l(s)$ are the arm-length and leg-length of the square
$s\in\la$.
In its simplest form, this conjecture can be stated as follows.
\begin{conjecture}\label{Conj_modqtNO}
For $r$ a positive integer
\begin{align*}
&\sum_{\la\in\PP} T^{\abs{\la}} S^{\abs{\HH_r(\la)}}
\prod_{\substack{s\in\la \\[1pt] h(s)\equiv 0 \Mod{r}}}
\frac{(1-uq^{a(s)+1}t^{l(s)})(1-u^{-1}q^{a(s)}t^{l(s)+1})}
{(1-q^{a(s)+1}t^{l(s)})(1-q^{a(s)}t^{l(s)+1})} \\[1mm]
&\quad\qquad=\frac{(T^r;T^r)_{\infty}^r}
{(T;T)_{\infty}(ST^r;ST^r)_{\infty}^r} 
\prod_{\substack{i,j\geq 1 \\ i+j\equiv 1 \Mod{r}}}
\frac{(uq^it^{j-1}ST^r,u^{-1}q^{i-1}t^jST^r;ST^r)_{\infty}}
{(q^it^{j-1}ST^r,q^{i-1}t^jST^r;ST^r)_{\infty}} \notag \\[1mm]
&\quad\qquad=\frac{(T^r;T^r)_{\infty}^r}
{(T;T)_{\infty}(ST^r;ST^r)_{\infty}^r} 
\prod_{i=1}^r 
\frac{(uq^it^{r-i}ST^r,u^{-1}q^{r-i}t^iST^r;q^r,t^r,ST^r)_{\infty}}
{(q^it^{r-i}ST^r,q^{r-i}t^iST^r;q^r,t^r,ST^r)_{\infty}}.
\end{align*}
\end{conjecture}

\medskip

The remainder of this paper is organised as follows.
After some preliminary discussions on formal power series and 
integer partitions in the next two sections, Section~\ref{Sec_Littlewood}
reviews Littlewood's classical decomposition of a partition into its
$r$-core and $r$-quotient. 
This is used in Section~\ref{Sec_Multiplication-thm} in our discussion
of the Han--Ji multiplication theorem.
Then, in Section~\ref{Sec_New-multiplication-thm}, we propose an analogue
of Littlewood's decomposition. 
This new decomposition is applied to prove an analogue of 
Theorem~\ref{Thm_HanJi} by Han and Ji.
Section~\ref{Sec_Applications} contains a number of applications
of this new multiplication theorem to hook-length formulas.
Finally, in Section~\ref{Sec_Conjectures} we present a number of 
conjectures and open problems.
This includes a refinement of Conjecture~\ref{Conj_modqtNO}, 
some problems pertaining to elliptic $q,t$-Nekrasov--Okounkov
formulas (Proposition~\ref{Prop_eqNO} of that section does 
prove an elliptic $q$-analogue of \eqref{Eq_NOmod})
and a discussion of a possible extension of our new multiplication
theorem motivated by a combinatorial identity of Buryak, Feigin and
Nakajima which arose in their work on quasihomogeneous Hilbert schemes.

\section{Formal power series}

All series and series identities considered in this paper are viewed from the
point of view of formal power series, typically in the formal variable $T$.
For example, we regard the Nekrasov--Okounkov formula \eqref{Eq_NO} as
an identity in $\mathbb{Q}[z][[T]]$, where 
\[
(1-T)^z:=\sum_{n\geq 0} \binom{z}{n} T^n
\]
and
\[
\binom{z}{n}:=(-1)^n \frac{z(z+1)\cdots(z+n-1)}{n!}.
\]
Those preferring an analytic point of view should have little trouble
adding the required convergence conditions. 
In the case of \eqref{Eq_NO}, for example, it suffices to take 
$T\in\mathbb{C}$, $z\in\mathbb{R}$ such that $\abs{T}<1$,
or $-1<T<1$ and $z\in\mathbb{C}$.

\section{Partitions}\label{Sec_Partitions}

A partition $\la=(\la_1,\la_2,\dots)$ is a weakly decreasing
sequence of nonnegative integers such that only finitely many
$\la_i$ are strictly positive.
The positive $\la_i$ are called the parts of $\la$, and the 
length of $\la$, denoted $\ell(\la)$, counts the number
of parts. 
If $\abs{\la}:=\sum_{i\geq 1}\la_i=n$ we say that $\la$ is a partition
of $n$, denoted as $\la\vdash n$.
We typically suppress the infinite tail of zeros of a partition,
so that, for example, the partition $(6,5,5,3,1,1,0,\dots)$ of
$21$ is denoted by $(6,5,5,3,1,1)$.
The set of all partitions, including the
unique partition of $0$ (also written as $0$), is denoted by $\PP$.
This set has the well-known generating function
\begin{equation}\label{Eq_GFpart}
\sum_{\la\in\PP} z^{\ell(\la)} T^{\abs{\la}}=\frac{1}{(zT;T)_{\infty}}.
\end{equation}

We identify a partition $\la$ with its Young diagram, consisting of
$\ell(\la)$ left-aligned rows of squares such that the $i$th row 
contains $\la_i$ squares.
For example, the partition $(6,5,5,3,1,1)$ corresponds to
the diagram
\medskip

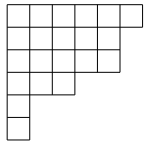
\begin{center}
\begin{tikzpicture}[scale=0.3,line width=0.3pt]
\draw (0,0)--(6,0);
\draw (0,-1)--(6,-1);
\draw (0,-2)--(5,-2);
\draw (0,-3)--(5,-3);
\draw (0,-4)--(3,-4);
\draw (0,-5)--(1,-5);
\draw (0,-6)--(1,-6);
\draw (0,0)--(0,-6);
\draw (1,0)--(1,-6);
\draw (2,0)--(2,-4);
\draw (3,0)--(3,-4);
\draw (4,0)--(4,-3);
\draw (5,0)--(5,-3);
\draw (6,0)--(6,-1);
\end{tikzpicture}
\end{center}
The squares of $\la$ are indexed by coordinates $(i,j)\in\mathbb{N}^2$,
with $i$ the row and $j$ the column coordinate, such that the 
top-left square corresponds to $(1,1)$.

The conjugate $\la'$ of the partition $\la$ is obtained by reflecting
$\la$ in the diagonal $i=j$, so that rows become columns and vice versa.
The conjugate of the partition in our running example is $(6,4,4,3,3,1)$.
Given a partition $\la$, the multiplicity of parts of size $i$,
denoted $m_i(\la)$, can be expressed in terms of $\la'$ as 
$m_i(\la)=\la'_i-\la'_{i+1}$.
We alternatively write partitions using the multiplicities, so that
$(6,4,4,3,3,1)=(6,4^2,3^2,1)$.

To each square $s=(i,j)\in\la$ we associate an arm, leg and hook,
defined as the sets of squares 
\begin{align*}
\arm(s)&:=\{(i,k): j<k\leq\la_i\},  \\[1mm] 
\leg(s)&:=\{(k,j): i<k\leq\la'_j\},  \\[1mm]
\hook(s)&:=\arm(s)\cup\leg(s)\cup\{s\}.
\end{align*}
Below, the arm and leg of the square $s=(2,2)$ of $(6,5,5,3,1,1)$ are 
marked in dark and light blue respectively in the diagram on the left. 
Similarly, the hook of $(2,2)$ is marked in the diagram on the right:
\smallskip

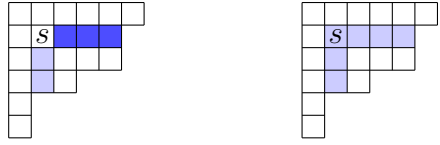
\begin{center}
\begin{tikzpicture}[scale=0.3,baseline=0cm]
\tikzmath{\x=13;}

\draw (1.5,-1.5) node {$s$};
\begin{scope}[color=blue!70]
\fill(2,-2) rectangle (5,-1);
\end{scope}
\begin{scope}[color=blue!20]
\fill(1,-2) rectangle (2,-4);
\end{scope}
\draw (0,0)--(6,0);
\draw (0,-1)--(6,-1);
\draw (0,-2)--(5,-2);
\draw (0,-3)--(5,-3);
\draw (0,-4)--(3,-4);
\draw (0,-5)--(1,-5);
\draw (0,-6)--(1,-6);
\draw (0,0)--(0,-6);
\draw (1,0)--(1,-6);
\draw (2,0)--(2,-4);
\draw (3,0)--(3,-4);
\draw (4,0)--(4,-3);
\draw (5,0)--(5,-3);
\draw (6,0)--(6,-1);
\begin{scope}[color=blue!20]
\fill(\x+1,-2) rectangle (\x+5,-1);
\fill(\x+1,-2) rectangle (\x+2,-4);
\end{scope}
\draw (\x,0)--(\x+6,0);
\draw (\x,-1)--(\x+6,-1);
\draw (\x,-2)--(\x+5,-2);
\draw (\x,-3)--(\x+5,-3);
\draw (\x,-4)--(\x+3,-4);
\draw (\x,-5)--(\x+1,-5);
\draw (\x,-6)--(\x+1,-6);
\draw (\x,0)--(\x,-6);
\draw (\x+1,0)--(\x+1,-6);
\draw (\x+2,0)--(\x+2,-4);
\draw (\x+3,0)--(\x+3,-4);
\draw (\x+4,0)--(\x+4,-3);
\draw (\x+5,0)--(\x+5,-3);
\draw (\x+6,0)--(\x+6,-1);
\draw (\x+1.5,-1.5) node {$s$};
\end{tikzpicture}
\end{center}

\medskip

\noindent
Correspondingly, we have the three statistics $a(s)$, $l(s)$ and $h(s)$,
known as arm-length, leg-length and hook-length, given by
\begin{align*}
a(s)&:=\abs{\arm(s)}=\la_i-j,\\[1mm]
l(s)&:=\abs{\leg(s)}=\la'_j-i, \\[1mm]
h(s)&:=\abs{\hook(s)}=\la_i+\la'_j-i-j+1.
\end{align*}
For $r$ a positive integer, the multiset
of hook-lengths of $\la$ congruent to $0$ modulo $r$
is denoted by $\HH_r(\la)$.
When $r=1$ we more simply write $\HH(\la)$ for the
multiset of all hook-lengths, omitting the subscript $1$.
It is often convenient to record $\HH_r(\la)$ or $\HH(\la)$
by writing the hook-lengths in the diagram of $\la$. 
For the partition in our example this gives

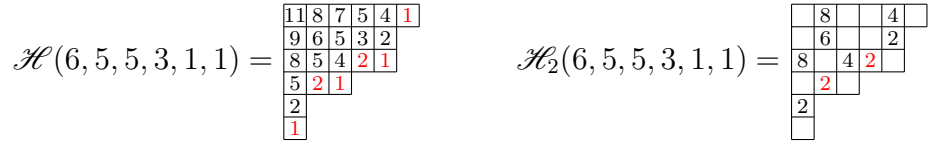
\begin{center}
\begin{tikzpicture}[scale=0.3,baseline=0cm]
\tikzmath{\x=22.5;}
\draw (0,0)--(6,0);
\draw (0,-1)--(6,-1);
\draw (0,-2)--(5,-2);
\draw (0,-3)--(5,-3);
\draw (0,-4)--(3,-4);
\draw (0,-5)--(1,-5);
\draw (0,-6)--(1,-6);
\draw (0,0)--(0,-6);
\draw (1,0)--(1,-6);
\draw (2,0)--(2,-4);
\draw (3,0)--(3,-4);
\draw (4,0)--(4,-3);
\draw (5,0)--(5,-3);
\draw (6,0)--(6,-1);
\draw (-6.25,-2.5) node {$\HH(6,5,5,3,1,1) = $};
\draw (0.5,-0.5) node {$\scriptstyle 11$};
\draw (1.5,-0.5) node {$\scriptstyle 8$};
\draw (2.5,-0.5) node {$\scriptstyle 7$};
\draw (3.5,-0.5) node {$\scriptstyle 5$};
\draw (4.5,-0.5) node {$\scriptstyle 4$};
\draw[red] (5.5,-0.5) node {$\scriptstyle 1$};
\draw (0.5,-1.5) node {$\scriptstyle 9$};
\draw (1.5,-1.5) node {$\scriptstyle 6$};
\draw (2.5,-1.5) node {$\scriptstyle 5$};
\draw (3.5,-1.5) node {$\scriptstyle 3$};
\draw (4.5,-1.5) node {$\scriptstyle 2$};
\draw (0.5,-2.5) node {$\scriptstyle 8$};
\draw (1.5,-2.5) node {$\scriptstyle 5$};
\draw (2.5,-2.5) node {$\scriptstyle 4$};
\draw[red] (3.5,-2.5) node {$\scriptstyle 2$};
\draw[red] (4.5,-2.5) node {$\scriptstyle 1$};
\draw (0.5,-3.5) node {$\scriptstyle 5$};
\draw[red] (1.5,-3.5) node {$\scriptstyle 2$};
\draw[red] (2.5,-3.5) node {$\scriptstyle 1$};
\draw (0.5,-4.5) node {$\scriptstyle 2$};
\draw[red] (0.5,-5.5) node {$\scriptstyle 1$};
\draw (\x,0)--(\x+6,0);
\draw (\x,-1)--(\x+6,-1);
\draw (\x,-2)--(\x+5,-2);
\draw (\x,-3)--(\x+5,-3);
\draw (\x,-4)--(\x+3,-4);
\draw (\x,-5)--(\x+1,-5);
\draw (\x,-6)--(\x+1,-6);
\draw (\x,0)--(\x,-6);
\draw (\x+1,0)--(\x+1,-6);
\draw (\x+2,0)--(\x+2,-4);
\draw (\x+3,0)--(\x+3,-4);
\draw (\x+4,0)--(\x+4,-3);
\draw (\x+5,0)--(\x+5,-3);
\draw (\x+6,0)--(\x+6,-1);
\draw (\x-6.35,-2.5) node {$\HH_2(6,5,5,3,1,1) = $};
\draw (\x+1.5,-0.5) node {$\scriptstyle 8$};
\draw (\x+4.5,-0.5) node {$\scriptstyle 4$};
\draw (\x+1.5,-1.5) node {$\scriptstyle 6$};
\draw (\x+4.5,-1.5) node {$\scriptstyle 2$};
\draw (\x+0.5,-2.5) node {$\scriptstyle 8$};
\draw (\x+2.5,-2.5) node {$\scriptstyle 4$};
\draw[red] (\x+3.5,-2.5) node {$\scriptstyle 2$};
\draw[red] (\x+1.5,-3.5) node {$\scriptstyle 2$};
\draw (\x+0.5,-4.5) node {$\scriptstyle 2$};
\end{tikzpicture}
\end{center}
\smallskip
\noindent where the colouring of some of the hook-lengths 
is to be ignored for now.

We refer to a square $s\in\la$ as a `bottom square' if it has
coordinates $(\la'_j,j)$ for some $1\leq j\leq\la_1$.
In other words, a square $s\in\la$ is a bottom square if it has 
leg-length $l(s)$ equal to zero. (In \cite{HN18} such squares are
referred to as having `trivial legs'.)
In analogy with $\HH_r(\la)$, we write $\HHb{r}(\la)$ for the 
multiset of hook-lengths of bottom squares that are congruent to 
$0$ modulo $r$, and set
$\HHbeen(\la):=\HHb{1}(\la)$.
The hook-lengths coloured red in the above two diagrams correspond to
those bottoms squares that contribute to $\HHbeen(6,5,5,3,1,1)$ and
$\HHb{2}(6,5,5,3,1,1)$ respectively.
Thus
\[
\HHbeen(6,5,5,3,1,1)=\{1^4,2^2\}\quad\text{and}\quad
\HHb{2}(6,5,5,3,1,1)=\{2^2\}.
\]
We also use the notation $\HHr{r}(\la):=\HHb{r}(\la')$ 
(r for right), so that $\HHreen(\la)$ is the multiset of hook lengths
of squares $s\in\la$ which have trivial arm.
We may think of
\begin{equation}\label{Eq_modular-length}
\ell_r(\la):=\abs{\HHr{r}(\la)}=\sum_{i\geq 1} 
\Big\lfloor\frac{m_i(\la)}{r}\Big\rfloor
\end{equation}
as a modular generalisation of the ordinary length statistic on partitions,
counting the number of squares $s=(i,\la_i)$ of $\la$ such that
$h(s)\equiv 0 \pmod{r}$.
For example, for the partition $\la=(3,2,2,1,1,1,1)$ we have 
$\ell(\la)=\ell_1(\la)=7$, $\ell_2(\la)=3$, 
$\ell_3(\la)=\ell_4(\la)=1$ and $\ell_r(\la)=0$ for $r\geq 5$:

\begin{center}
\begin{tikzpicture}[scale=0.3,line width=0.3pt]
\draw (0,0)--(3,0);
\draw (0,-1)--(3,-1);
\draw (0,-2)--(2,-2);
\draw (0,-3)--(2,-3);
\draw (0,-4)--(1,-4);
\draw (0,-5)--(1,-5);
\draw (0,-6)--(1,-6);
\draw (0,-7)--(1,-7);
\draw (0,0)--(0,-7);
\draw (1,0)--(1,-7);
\draw (2,0)--(2,-3);
\draw (3,0)--(3,-1);
\draw (2.5,-0.5) node {\red{$\scriptstyle 1$}};
\draw (1.5,-1.5) node {\red{$\scriptstyle 2$}};
\draw (1.5,-2.5) node {\red{$\scriptstyle 1$}};
\draw (0.5,-3.5) node {\red{$\scriptstyle 4$}};
\draw (0.5,-4.5) node {\red{$\scriptstyle 3$}};
\draw (0.5,-5.5) node {\red{$\scriptstyle 2$}};
\draw (0.5,-6.5) node {\red{$\scriptstyle 1$}};
\end{tikzpicture}
\end{center}

A partition $\la$ is called an $r$-core if $\HH_r(\la)=\emptyset$,
i.e., if none of its hook-lengths is a multiple of $r$. 
We use $\CC_r$ to denote the set of $r$-cores.
Note that $\CC_1=\{0\}$ and $\CC_2$ is the set of
staircase partitions: 
\[
\CC_2=\{\delta_n: n\geq 1\},
\]
where $\delta_n:=(n-1,\dots,2,1,0)$.
In much the same way, we say that $\la$ is an $r$-kernel if
$\HHb{r}(\la)=\emptyset$, and denote the set of $r$-kernels
by $\KK_r$. Clearly, $\CC_r\subseteq\KK_r$,
where the inclusion is strict unless $r=1$.
$\KK_r$ is given by the set of partitions $\la$ such that the
differences between consecutive $\la_i$ are at most $r-1$:
\begin{equation}\label{Eq_rkernels}
\KK_r=\big\{\la\in\mathscr{P}: 
\la_i-\la_{i+1}<r \text{ for all } i\geq 1\big\}.
\end{equation}
It is an elementary fact, see e.g., \cite{Andrews76}, that
such partitions have generating function
\begin{equation}\label{Eq_GF-rkernels}
\sum_{\la\in\KK_r} z^{\la_1} T^{\abs{\la}}=
\frac{(z^rT^r;T^r)_{\infty}}{(zT;T)_{\infty}}.
\end{equation}

\section{Littlewood's decomposition}\label{Sec_Littlewood}

Littlewood's decomposition 
\begin{align*}
\phi_r:\PP &\longrightarrow \CC_r\times \PP^r \\
\la &\longmapsto (\mu,\boldsymbol{\nu})=
\big(\mu,(\nu^{(0)},\nu^{(1)},\dots,\nu^{(r-1)})\big)
\end{align*}
is a generalisation of Euclidean division to integer partitions, 
and first arose in the modular representation theory of the symmetric 
group \cite{Littlewood51,Nakayama41}.
Given a positive integer $r$, it decomposes a partition
$\la$ into an $r$-core, $\mu$, and a sequence 
$\boldsymbol{\nu}=(\nu^{(0)},\nu^{(1)},\dots,\nu^{(r-1)})$ of $r$ 
partitions, known as the $r$-quotient of $\la$.
Instead of $\mu$ we will sometimes write $\core{\la}$ for the 
$r$-core of $\la$.
We also use the shorthand notation
\[
\abs{\boldsymbol{\nu}}:=\sum_{i=0}^{r-1} \abs{\nu^{(i)}} \quad
\text{and}\quad
\HH(\boldsymbol{\nu}):=\bigcup_{i=0}^{r-1} \HH(\nu^{(i)}),
\]
where the union is that of multisets.

There are numerous equivalent descriptions of $\phi_r$, see
e.g., \cite{AF04,GKS90,HJ11,JK81,vanLeeuwen99,Macdonald95,Tingley08}.
Given a partition $\la$ we form its bi-infinite edge or $0/1$-sequence 
$s=s(\la)$ (also known as the code of $\la$) by tracing the extended 
boundary of $\la$, encoding an up step by a $0$ and a right step by a $1$.
(When $0$s are replaced by black beads and $1$s by white beads, such a
sequence is also known as a Maya diagram \cite[\S 4.1]{DJKMO89}.)
For example, the $0/1$-sequence of the partition $(5,4,4,1)$ is determined
as
\medskip
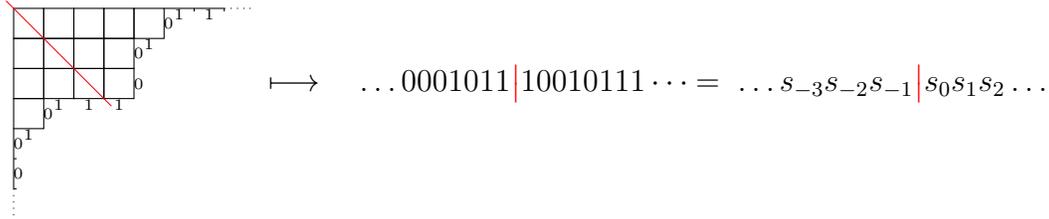
\begin{center}
\begin{tikzpicture}[scale=0.4,baseline=0cm,line width=1pt]
\draw[thin] (0,0)--(0,-2);
\draw[thin] (0,-1)--(0.1,-1);
\draw[thin] (0,-2)--(0.1,-2);
\draw[thin,dotted] (0,-2)--(0,-3);
\draw[thin] (5,4)--(7,4);
\draw[thin] (6,4)--(6,3.9);
\draw[thin] (7,4)--(7,3.9);
\draw[thin,dotted] (7,4)--(8,4);
\draw[thin] (0,0) rectangle (1,1);
\draw[thin] (0,1) rectangle (1,2);
\draw[thin] (0,2) rectangle (1,3);
\draw[thin] (0,3) rectangle (1,4);
\draw[thin] (1,1) rectangle (2,2);
\draw[thin] (1,2) rectangle (2,3);
\draw[thin] (1,3) rectangle (2,4);
\draw[thin] (2,1) rectangle (3,2);
\draw[thin] (2,2) rectangle (3,3);
\draw[thin] (2,3) rectangle (3,4);
\draw[thin] (3,1) rectangle (4,2);
\draw[thin] (3,2) rectangle (4,3);
\draw[thin] (3,3) rectangle (4,4);
\draw[thin] (4,3) rectangle (5,4);
\draw[thin,red] (-0.25,4.25)--(3.25,0.75);
\draw (0.15,-1.5) node {$\scriptscriptstyle{0}$};
\draw (0.15,-0.5) node {$\scriptscriptstyle{0}$};
\draw (0.5,-0.2) node {$\scriptscriptstyle{1}$};
\draw (1.15,0.5) node {$\scriptscriptstyle{0}$};
\draw (1.5,0.8) node {$\scriptscriptstyle{1}$};
\draw (2.5,0.8) node {$\scriptscriptstyle{1}$};
\draw (3.5,0.8) node {$\scriptscriptstyle{1}$};
\draw (4.15,1.5) node {$\scriptscriptstyle{0}$};
\draw (4.15,2.5) node {$\scriptscriptstyle{0}$};
\draw (4.5,2.8) node {$\scriptscriptstyle{1}$};
\draw (5.15,3.5) node {$\scriptscriptstyle{0}$};
\draw (5.5,3.8) node {$\scriptscriptstyle{1}$};
\draw (6.5,3.8) node {$\scriptscriptstyle{1}$};
\draw (21.5,1.5) node {$\longmapsto\quad\dots
0001011\red{\big|}10010111\dots = \; \dots 
s_{-3}s_{-2}s_{-1}\red{\big|}s_0s_1s_2\dots$};
\end{tikzpicture}
\end{center}

\noindent
Here we have put a marker in the $0/1$-sequence such that the number of
ones to the left of the marker is equal to the number of zeros to its
right. Equivalently, the marker corresponds to the $0/1$-sequence
crossing the main diagonal of the partition.
The $r$ subsequences $s^{(0)},\dots,s^{(r-1)}$ defined by
\[
s^{(i)}:=( s_{i+rj})_{j\in\mathbb{Z}}
=(\dots s_{i-2r}s_{i-r}\red{\big|}s_is_{i+r}\dots)
\quad\text{for $0\leq i\leq r-1$}
\]
correspond to the $0/1$-sequences of 
$\nu^{(0)},\dots,\nu^{(r-1)}$ forming the $r$-quotient $\boldsymbol{\nu}$.
Note that we have left the marker in its original position so that
the balancing of zeros and ones will generally not hold for the
individual $s^{(i)}$.
For example, when $r=3$ the subsequences $s^{(0)},s^{(1)},s^{(2)}$
and partitions $\nu^{(0)},\nu^{(1)},\nu^{(2)}$ corresponding to
the partition $(5,4,4,1)$ are given by
\begin{align*}
s^{(0)} & = \dots 0000\red{\big|}1111\dots & \nu^{(0)}&=0 \\
s^{(1)} & = \dots 0001\red{\big|}0011\dots & \nu^{(1)}&=(1,1) \\
s^{(2)} & = \dots 0011\red{\big|}0111\dots & \nu^{(2)}&=(2).
\end{align*}
To also obtain the $r$-core of $\la$ we move the zeros
in each of the $s^{(i)}$ to the left, like the beads on an 
abacus \cite{JK81}, and then reassemble the subsequences to form 
a single $0/1$-sequence. For our example this gives
\begin{align*}
s^{(0)}&=\dots 0000\red{\big|}1111\dots 
\mapsto \dots0000\red{\big|}1111\dots \\
s^{(1)}&=\dots 0001\red{\big|}0011\dots 
\mapsto \dots0000\red{\big|}0111\dots \;\longmapsto\;
 \dots0001\red{\big|}10111\dots\\
s^{(2)}&=\dots 0011\red{\big|}0111\dots \mapsto 
\dots0001\red{\big|}1111\dots
\end{align*}
so that the $3$-core of $(5,4,4,1)$ is $(2)$. Hence
\[
\phi_3(5,4,4,1)=\Big( (2),\big(0,(1,1),(2)\big)\Big).
\]
Alternatively, if we only interested in finding the $r$-core of $\la$,
we may colour the $0/1$-sequence $s(\la)$ with $r$ colours according
to the $r$ congruence classes formed by the position labels. 
Then we push all of the zeros of each of the $r$ colours to the left, past the
ones of that same colour, to obtain the $0/1$-sequence of its $r$-core.
In the case of our example this would give
\begin{multline}\label{Eq_colouredexample}
\dots 01011 \red{\big|}1001011\ldots \,\cong\,
\dots 
\blue{0}1\magenta{0}\blue{1}1 \red{\big|}
\magenta{1}\blue{0}0\magenta{1}\blue{0}1\magenta{1} 
\dots  \\
\stackrel{\phi_3}{\longmapsto}\,
\blue{0}0\magenta{0}\blue{0}1 \red{\big|}
\magenta{1}\blue{0}1\magenta{1}\blue{1}1\magenta{1}
\dots \, \cong \, \dots 00001 \red{\big|} 1011111\dots 
\end{multline}
It follows that if $\la$ and $\eta$ are partitions such that 
$\la_i\equiv\eta_i \pmod{r}$ then $\core{\la}=\core{\eta}$.\footnote{The
converse is not true, and for $\la,\eta\in\PP$ such that 
$\max\{\ell(\la),\ell(\eta)\}\leq n$, $\core{\la}=\core{\eta}$ 
if and only if $\la\equiv w(\eta+\delta_n)-\delta_n \pmod{r}$ for some
$w\in\mathfrak{S}_n$, see \cite[page 13]{Macdonald95}.} 
If $\la$ and $\eta$ are two partitions such that $\la_i=\eta_i$ for
all $i\neq j$ for some fixed $j$ and $\eta_j=\la_j-r$, then the $0/1$-sequences
of $\la$ and $\eta$ differ only in two places, a distance $r$ apart:
\[
\la: \dots \blue{0}\underbrace{\magenta{1}1\dots \green{1}
\blue{1}}_{r \text{ ones}} \dots,\qquad
\eta: \dots \underbrace{\blue{1}\magenta{1}1\dots 
\green{1}}_{r \text{ ones}}\blue{0} \dots,
\]
where the `dots' to the left and right are the same for $\la$ and $\eta$.
Hence $\core{\la}=\core{\eta}$.

Some key properties of Littlewood's decomposition \eqref{Eq_Litt-decomp}
are collected in the following proposition.
\begin{proposition}
Littlewood's decomposition
\begin{subequations}\label{Eq_Litt-decomp}
\begin{align}
\phi_r:\PP &\longrightarrow \CC_r\times \PP^r \\
\la &\longmapsto (\mu,\boldsymbol{\nu})
\end{align}
\end{subequations}
is a bijection such that
\begin{equation}\label{Eq_Lit2}
\abs{\la}=\abs{\mu}+r\,\abs{\boldsymbol{\nu}}
\end{equation}
and
\begin{equation}\label{Eq_Lit3}
\HH_{r}(\la)=r\HH(\boldsymbol{\nu}),
\end{equation}
where, for a set or multiset $S$, 
$rS:=\{rs: s\in S\}$.
\end{proposition}

In the case of the above example the respective hook-lengths are

\begin{center}
\begin{tikzpicture}[scale=0.3,baseline=0cm]
\tikzmath{\x=24;}

\draw (0,0)--(5,0);
\draw (0,-1)--(5,-1);
\draw (0,-2)--(4,-2);
\draw (0,-3)--(4,-3);
\draw (0,-4)--(1,-4);
\draw (0,0)--(0,-4);
\draw (1,0)--(1,-4);
\draw (2,0)--(2,-3);
\draw (3,0)--(3,-3);
\draw (4,0)--(4,-3);
\draw (5,0)--(5,-1);
\draw (-5.1,-1.5) node {$\HH_3(5,4,4,1) = $};
\draw (1.5,-0.5) node {$\scriptstyle 6$};
\draw (0.5,-1.5) node {$\scriptstyle 6$};
\draw (2.5,-1.5) node {$\scriptstyle 3$};
\draw (1.5,-2.5) node {$\scriptstyle 3$};
\draw (\x,-1.5) node {$\emptyset$\;,\hspace{6mm},};
\draw (\x,-0.5)--(\x+1,-0.5);
\draw (\x,-1.5)--(\x+1,-1.5);
\draw (\x,-2.5)--(\x+1,-2.5);
\draw (\x,-0.5)--(\x,-2.5);
\draw (\x+1,-0.5)--(\x+1,-2.5);
\draw (\x+2.5,-1)--(\x+4.5,-1);
\draw (\x+2.5,-2)--(\x+4.5,-2);
\draw (\x+2.5,-1)--(\x+2.5,-2);
\draw (\x+3.5,-1)--(\x+3.5,-2);
\draw (\x+4.5,-1)--(\x+4.5,-2);
\draw (\x-8.3,-1.5) node {$\HH\big(0,(1,1),(2)\big) = \Big($};
\draw (\x+5.3,-1.5) node {$\Big)$,};
\draw (\x+0.5,-1) node {$\scriptstyle 2$};
\draw (\x+0.5,-2) node {$\scriptstyle 1$};
\draw (\x+3,-1.5) node {$\scriptstyle 2$};
\draw (\x+4,-1.5) node {$\scriptstyle 1$};
\end{tikzpicture}
\end{center}
\smallskip
consistent with \eqref{Eq_Lit3}.

It follows from the properties of $\phi_r$, in particular bijectivity 
and \eqref{Eq_Lit2}, that 
\begin{equation}\label{Eq_GFPPrmu}
\sum_{\substack{\la\in\PP \\[1pt] \core{\la}=\omega}} T^{\abs{\la}}=
\frac{T^{\abs{\omega}}}{(T^r;T^r)_{\infty}^r},
\end{equation}
where $\omega\in\CC_r$. 
Summing both sides over $\omega\in\CC_r$, the left-hand side becomes the 
ordinary generating function for partitions. Solving
for $\sum_{\omega\in\CC_r} T^{\abs{\omega}}$
then yields \cite{Klyachko82}
\begin{equation}\label{Eq_GF-rcores}
\sum_{\omega\in\CC_r} T^{\abs{\omega}}=
\frac{(T^r;T^r)_{\infty}^r}{(T;T)_{\infty}}.
\end{equation}

\section{The multiplication theorem of Han and Ji}\label{Sec_Multiplication-thm}

The Han--Ji multiplication theorem, stated as Theorem~\ref{Thm_HanJi}
in the introduction, provides a simple mechanism to obtain modular
analogues of many known hook-length formulas.
For example, by combining the hook-length formula \eqref{Eq_FRT} with the 
Robinson--Schensted correspondence \cite{Robinson38,Schensted61} between
permutations and pairs of standard Young tableaux, it follows 
that\footnote{To avoid RS one may alternatively use that 
(i) $f^{\la}$ gives the dimension of the irreducible (complex) 
$\mathfrak{S}_n$-module indexed by the partition $\la$, 
(ii) the sum of the squares of the dimensions of the irreducible
$G$-modules of a finite group $G$ is equal to the order of $G$.
Hence $\sum_{\la\vdash n} (f^{\la})^2=n!$, which, by \eqref{Eq_FRT},
is equivalent to \eqref{Eq_hook}.}
\begin{equation}\label{Eq_hook}
\sum_{\la\in\PP} T^{\abs{\la}} \prod_{h\in\HH(\la)}\frac{1}{h^2}
=\eup^T.
\end{equation}
After the substitution $T\mapsto T/r^2$ this takes the form 
\[
\sum_{\la\in\PP} T^{\abs{\la}} \prod_{h\in\HH(\la)}\frac{1}{(rh)^2},
=\eup^{T/r^2}
\]
Comparing this with \eqref{Eq_seed} it follows that for $\rho(h)=1/h^2$
we have the closed-form expression for $f_r(T)$ given by
$f_r(T)=\exp(T/r^2)$.
By \eqref{Eq_trafo} we thus obtain the modular analogue
(see \cite[Corollary 5.4]{Han10})
\begin{equation}\label{Eq_hookp}
\sum_{\la\in\PP} T^{\abs{\la}} \prod_{h\in\HH_r(\la)}\frac{S}{h^2}
=\eup^{ST^r/r}\frac{(T^r;T^r)_{\infty}^r}{(T;T)_{\infty}}.
\end{equation}
Similarly, by replacing $z\mapsto z/r^2$ in \eqref{Eq_NO}, 
Han's modular analogue of the Nekrasov--Okounkov formula \eqref{Eq_NOmod}
follows from the multiplication theorem with
\[
\rho(h)=1-\frac{z}{h^2} \quad\text{and}\quad f_r(T)=(T;T)_{\infty}^{z/r^2-1}.
\]

\medskip

The proof of Theorem~\ref{Thm_HanJi} is surprisingly simple,
and follows in a few elementary steps from Littlewood's decomposition.
Using all of \eqref{Eq_Litt-decomp}--\eqref{Eq_Lit3} 
(with $\mu\mapsto\omega$) and also noting that, by \eqref{Eq_Lit3}, 
\[
S^{\abs{\HH_r(\la)}}=S^{\abs{\boldsymbol{\nu}}}=
S^{\sum_{i=0}^{r-1} \abs{\nu^{(i)}}},
\]
we get
\[
\sum_{\substack{\la\in\PP \\[1pt] \core{\la}=\omega}}  
T^{\abs{\la}} S^{\abs{\HH_r(\la)}} 
\prod_{h\in \HH_r(\la)} \rho(h) 
=T^{\abs{\omega}}
\Bigg( \sum_{\nu\in\PP} (ST^r)^{\abs{\nu}} \prod_{h\in \HH(\nu)} 
\rho(rh)\Bigg)^r
\]
for $\omega\in\CC_r$.
By \eqref{Eq_seed} with $(T,a)\mapsto(ST^r,r)$ this simplifies to
\begin{equation}\label{Eq_fix-rcore}
\sum_{\substack{\la\in\PP \\[1pt] \core{\la}=\omega}}  
T^{\abs{\la}} S^{\abs{\HH_r(\la)}} 
\prod_{h\in\HH_r(\la)} \rho(h)
=T^{\abs{\omega}} \big(f_r(ST^r)\big)^r.
\end{equation}
Summing $\omega$ over $\CC_r$, and using \eqref{Eq_GF-rcores}
for the generating function of $r$-cores, \eqref{Eq_trafo} follows.

If we divide both sides of \eqref{Eq_fix-rcore} by $T^{\abs{\omega}}$,
use that $\abs{\HH_r(\la)}=(\abs{\la}-\abs{\omega})/r$ and finally
replace $ST^r$ by $T$,
we obtain a variant of the multiplication theorem without the factor
\[
\frac{(T^r;T^r)_{\infty}^r}{(T;T)_{\infty}}.
\]

\begin{theorem}[Modified Han--Ji multiplication theorem]\label{Thm_HanJi-mod}
For $r$ a positive integer and $\rho$ a function on the positive integers,
let $f_r(T)$ be defined by \eqref{Eq_seed}. Then
\[
\sum_{\substack{\la\in\PP \\[1pt] \core{\la}=\omega}} 
T^{(\abs{\la}-\abs{\omega})/r} \prod_{h\in \HH_r(\la)} \rho(h)
=\big(f_r(T)\big)^r,
\]
where $\omega$ is an $r$-core.
\end{theorem}

The reason for stating this alternative version is that in 
Section~\ref{Sec_qt-modular} we discuss some modular 
Nekrasov--Okounkov-type series of the form
\[
\sum_{\substack{\la\in\PP \\[1pt] \core{\la}=\omega}} 
T^{(\abs{\la}-\abs{\omega})/r} \rho_r(\la)
\]
for which the observed `$\omega$-independence' of the sum cannot 
simply be explained by the Littlewood decomposition.

\section{An analogue of the Han--Ji multiplication theorem}\label{Sec_New-multiplication-thm}

In this section we describe a new Littlewood-like decomposition which
implies the following analogue of the Theorem~\ref{Thm_HanJi} 

\begin{subtheorem}{theorem}\label{Thm_new}
\begin{theorem}\label{Thm_multiplication-new}
For $r$ a positive integer and $\rho$ a function on the positive integers,
let $f_r(T)$ be the formal power series defined by
\begin{equation}\label{Eq_input}
f_r(T):=\sum_{\la\in\PP} T^{\abs{\la}} 
\prod_{h\in\HHbeen(\la)} \rho(rh).
\end{equation}
Then
\begin{equation}\label{Eq_output}
\sum_{\substack{\la\in\PP \\[1pt] \core{\la}=\omega}} 
T^{(\abs{\la}-\abs{\omega})/r} 
\prod_{h\in\HHb{r}(\la)} \rho(h) = 
\frac{f_r(T)}{(T;T)_{\infty}^{r-1}},
\end{equation}
where $\omega$ is an $r$-core.
\end{theorem}

Equation \eqref{Eq_output} may be replaced by an expression
which includes a factor representing $r$-cores. This is to be
compared with Theorem~\ref{Thm_HanJi}.

\begin{theorem}\label{Thm_multiplication-new-2}
For $r$ a positive integer and $\rho$ a function on the positive integers,
let $f_r(T)$ be defined by \eqref{Eq_input}. Then
\begin{equation}\label{Eq_output-mod}
\sum_{\la\in\PP}T^{\abs{\la}} S^{\abs{\HH_r(\la)}}
\prod_{h\in\HHb{r}(\la)} \rho(h) 
=\frac{(T^r;T^r)_{\infty}^r}{(T;T)_{\infty}(ST^r;ST^r)_{\infty}^{r-1}}
\,f_r(ST^r).
\end{equation}
\end{theorem}
\end{subtheorem}

\medskip

Let $\{\cdot\}:\mathbb{R}\to [0,1)$ and 
$\floor{\cdot}:\mathbb{R}\to\mathbb{Z}$ be the fractional-part and
floor functions respectively, and recall that $\KK_r$ is the set
of $r$-kernels, see \eqref{Eq_rkernels}.
We define the Littlewood-like map
\begin{align*}
\psi_r:\PP&\longrightarrow \KK_r\times\PP \\
\la &\longmapsto (\mu,\nu)
\end{align*}
by 
\begin{subequations}
\begin{align}
\label{Eq_numu}
\mu_i&=r\sum_{j\geq i} \{(\la_j-\la_{j+1})/r\} \\
\nu_i&=\sum_{j\geq i} \floor{(\la_j-\la_{j+1})/r}
\label{Eq_omegamu}
\end{align}
\end{subequations}
for all $i\geq 1$.
Since $\la$ is a partition, it is clear that both $\mu$ and $\nu$ are 
partitions. Moreover, since
\[
\mu_i-\mu_{i+1}=r\{(\la_i-\la_{i+1})/r\}\in\{0,1,\dots,r-1\},
\]
it follows that $\mu$ is an $r$-kernel.

Graphically, $\mu$ and $\nu$ are obtained from $\la$ by identifying
the bottom squares of $\la$ with hook-lengths congruent to $0$ modulo $r$,
and colouring the columns spanned by these squares as well as the $r-1$ 
columns immediately to the right of these.
The coloured squares of $\la$ then yield $r\,\nu$ and the
remaining white squares form $\mu$.
For example, if $r=3$ and $\la=(14,6,6,1)$, then $\mu=(5,3,3,1)$
and $\nu=(3,1,1)$:

\medskip

\begin{center}
\tikzmath{\x=13;}
\tikzmath{\y=13;}

\begin{tikzpicture}[scale=0.3,line width=0.3pt]
\begin{scope}[color=blue!20]
\fill(8,-1) rectangle (14,0);
\fill(3,-3) rectangle (6,0);
\end{scope}
\begin{scope}[color=blue!70]
\fill(11,-1) rectangle (12,0);
\fill(8,-1) rectangle (9,0);
\fill(3,-3) rectangle (4,-2);
\end{scope}
\draw (0,0)--(14,0);
\draw (0,-1)--(14,-1);
\draw (0,-2)--(6,-2);
\draw (0,-3)--(6,-3);
\draw (0,-4)--(1,-4);
\draw (0,0)--(0,-4);
\draw (1,0)--(1,-4);
\draw (2,0)--(2,-3);
\draw (3,0)--(3,-3);
\draw (4,0)--(4,-3);
\draw (5,0)--(5,-3);
\draw (6,0)--(6,-3);
\draw (7,0)--(7,-1);
\draw (8,0)--(8,-1);
\draw (9,0)--(9,-1);
\draw (10,0)--(10,-1);
\draw (11,0)--(11,-1);
\draw (12,0)--(12,-1);
\draw (13,0)--(13,-1);
\draw (14,0)--(14,-1);
\draw (-1.7,-1.75) node {$\la=$};
\draw (16.4,-2) node {$\longmapsto$};
\draw (19.3,-2) node {$\mu=$};
\draw (21,0)--(26,0);
\draw (21,-1)--(26,-1);
\draw (21,-2)--(24,-2);
\draw (21,-3)--(24,-3);
\draw (21,-4)--(22,-4);
\draw (21,0)--(21,-4);
\draw (22,0)--(22,-4);
\draw (23,0)--(23,-3);
\draw (24,0)--(24,-3);
\draw (25,0)--(25,-1);
\draw (26,0)--(26,-1);
\draw (28.5,-2) node {$,\hspace{3.0mm} \nu=$};
\begin{scope}[color=blue!70]
\fill(32,-1) rectangle (34,0);
\fill(31,-3) rectangle (32,-2);
\end{scope}
\draw (31,0)--(34,0);
\draw (31,-1)--(34,-1);
\draw (31,-2)--(32,-2);
\draw (31,-3)--(32,-3);
\draw (31,0)--(31,-3);
\draw (32,0)--(32,-3);
\draw (33,0)--(33,-1);
\draw (34,0)--(34,-1);
\end{tikzpicture}
\end{center}
where the bottom squares of $\la$ with hook-lengths congruent to
$0$ modulo $3$ as well as the bottom squares of $\nu$ have
been marked in dark blue.

Although $\psi_r$ is much simpler than Littlewood's map $\phi_r$,
it has many properties in common with the latter.

\begin{proposition}\label{Prop_properties}
The decomposition 
\begin{subequations}\label{Eq_Litt-decomp-new}
\begin{align}
\psi_r:\PP&\longrightarrow \KK_r\times\PP \\
\la &\longmapsto (\mu,\nu)
\end{align}
\end{subequations}
defined by \eqref{Eq_numu} and \eqref{Eq_omegamu} is a bijection 
such that
\begin{align}
\core{\la}&=\core{\mu}, \label{Eq_samecore} \\[2mm]
\abs{\la}&=\abs{\mu}+r\abs{\nu}, \label{Eq_size-lamurnu} \\[2mm]
\label{Eq_HHtorHH}
\HHb{r}(\la)&=r\HHbeen(\nu) .
\end{align}
\end{proposition}

Equations \eqref{Eq_size-lamurnu} and \eqref{Eq_HHtorHH} are the 
analogues of \eqref{Eq_Lit2} and \eqref{Eq_Lit3} in Littlewood's
decomposition. Equation \eqref{Eq_samecore} also holds
in the Littlewood case as part of the much stronger 
$\core{\la}=\core{\mu}=\mu$.

\begin{proof}
To see that $\psi_r$ is a bijection we use
\[
r\{n/r\}=n-r\floor{n/r}
\]
to rewrite \eqref{Eq_numu} as
\begin{align*}
\mu_i &=\sum_{j\geq i}\Big(\la_j-\la_{j+1}-r\floor{(\la_j-\la_{j+1})/r}\Big) \\
&=\la_i-r\sum_{j\geq i}\floor{(\la_j-\la_{j+1})/r} \\
&=\la_i-r\nu_i,
\end{align*}
where the final equality follows from \eqref{Eq_omegamu}.
This shows that the tuple $(\mu,\nu)$ uniquely fixes $\la$,
so that $\psi_r$ is injective. Surjectivity is also clear 
since, for arbitrary $\mu\in\KK_r$ and $\nu\in\PP$, the image of 
$\mu+r\nu:=(\mu_1+r\nu_1,\mu_2+r\nu_2,\dots)$ is $(\mu,\nu)$.
We may thus conclude that $\psi_r$ is a bijection, with inverse
\begin{align*}
\psi_r^{-1}:
\KK_r\times\PP &\longrightarrow \PP \\
(\mu,\nu) &\longmapsto \la
\end{align*}
given by 
\[
\la_i=\mu_i+r\nu_i \quad\text{for all $i\geq 1$}.
\]
This immediately implies \eqref{Eq_size-lamurnu} and, recalling the
discussion following equation \eqref{Eq_colouredexample}, it also shows
\eqref{Eq_samecore}.

Finally, since $\mu\in\KK_r$,
\[
\HHb{r}(\la)=\HHb{r}(\la-\mu)=\HHb{r}(r\nu)=r\HHb{1}(\nu)
=r\HHbeen(\nu),
\]
completing the proof.
\end{proof} 

The generating function of $r$-kernels with fixed $r$-core 
admits a closed-form expression as follows.

\begin{lemma}\label{Lem_GF-rkernels-omega}
Let $\omega\in\CC_r$. Then
\[
\sum_{\substack{\mu\in\KK_r \\[1pt] \core{\mu}=\omega}} T^{\abs{\mu}}
=\frac{T^{\abs{\omega}}}{(T^r;T^r)_{\infty}^{r-1}}.
\]
\end{lemma}

Summing the left-hand side over $\omega\in\CC_r$ yields the generating 
function for all $r$-kernels. 
By \eqref{Eq_GF-rcores}, carrying out this  same sum on the right
yields $(T^r;T^r)_{\infty}/(T;T)_{\infty}$, so that we recover
\eqref{Eq_GF-rkernels}.

\begin{proof}
Let $\omega\in\CC_r$.
From Proposition~\ref{Prop_properties} it follows that
\[
\sum_{\substack{\la\in\PP \\[1pt] \core{\la}=\omega}}T^{\abs{\la}}=
\frac{1}{(T^r;T^r)_{\infty}}
\sum_{\substack{\mu\in\KK_r \\[1pt] \core{\mu}=\omega}} T^{\abs{\mu}}.
\]
The left-hand side is the generating function of partitions with
fixed $r$-core, which can be expressed in closed form by \eqref{Eq_GFPPrmu}.
Multiplying both sides by $(T^r;T^r)_{\infty}$ the claim follows.
\end{proof}

We now have all the ingredients needed to prove 
Theorems~\ref{Thm_multiplication-new} and \ref{Thm_multiplication-new-2}.

\begin{proof}\label{page_pf}
Let $\omega\in\CC_r$.
By Proposition~\ref{Prop_properties},
\[
\sum_{\substack{\la\in\PP \\[1pt] \core{\la}=\omega}}
T^{\abs{\la}} \prod_{h\in\HHb{r}(\la)} \rho(h)
=\Bigg(\,\sum_{\substack{\mu\in\KK_r \\[1pt] \core{\mu}=\omega}} 
T^{\abs{\mu}}\Bigg)
\Bigg(\sum_{\nu\in\PP} T^{r\abs{\nu}} 
\prod_{h\in\HHbeen(\nu)} \rho(rh)\Bigg).
\]
The first sum on the right can be carried out by
Lemma~\ref{Lem_GF-rkernels-omega}, whereas the second sum is exactly
$f_r(T^r)$ by \eqref{Eq_input} with $T\mapsto T^r$.
Hence
\[
\sum_{\substack{\la\in\PP \\[1pt] \core{\la}=\omega}}
T^{\abs{\la}} \prod_{h\in\HHb{r}(\la)} \rho(h) 
=\frac{T^{\abs{\omega}}}{(T^r;T^r)_{\infty}^{r-1}}
\,f_r(T^r).
\]
Dividing both sides by $T^{\abs{\omega}}$ and replacing $T$ by $T^{1/r}$
yields \eqref{Eq_output}. If instead we replace $T$ by $TS^{1/r}$,
then multiply both sides by $S^{-\abs{\omega}/r}$ and finally sum over 
$\mu\in\CC_r$ using \eqref{Eq_GF-rcores}, we obtain 
\[
\sum_{\omega\in\CC_r}
\sum_{\substack{\la\in\PP \\[1pt] \core{\la}=\omega}}
T^{\abs{\la}} S^{(\abs{\la}-\abs{\omega})/r} 
\prod_{h\in\HHb{r}(\la)} \rho(h)
=\frac{(T^r;T^r)_{\infty}^r}{(T;T)_{\infty}(ST^r;ST^r)_{\infty}^{r-1}}
\,f_r(ST^r).
\]
From \eqref{Eq_Lit2} and \eqref{Eq_Lit3} it follows that
$(\abs{\la}-\abs{\omega})/r=\abs{\HH_r(\la)}$, so that the left-hand side
may be replaced by
\[
\sum_{\omega\in\CC_r}
\sum_{\substack{\la\in\PP \\[1pt] \core{\la}=\omega}}
T^{\abs{\la}} S^{\abs{\HH_r(\la)}} \prod_{h\in\HHb{r}(\la)} \rho(h) \\
=\sum_{\la\in\PP}
T^{\abs{\la}} S^{\abs{\HH_r(\la)}} \prod_{h\in\HHb{r}(\la)} \rho(h) .
\]
This also proves \eqref{Eq_output-mod}.
\end{proof}

\section{Applications}\label{Sec_Applications}

While hook-length formulas abound in the combinatorics literature, 
identities that involve only hook-lengths of bottom squares are rare, 
making it more difficult to apply Theorems~\ref{Thm_multiplication-new}
and \ref{Thm_multiplication-new-2} than the Han--Ji multiplication 
theorem.

As a first example we discuss what is essentially a trivial application 
by taking $\rho(h)=z$, independent of $h$.
Then the left-hand side of \eqref{Eq_input} simplifies to
\[
\sum_{\la\in\PP} T^{\abs{\la}} z^{\abs{\HHbeen(\la)}}
=\sum_{\la\in\PP} T^{\abs{\la}} z^{\la_1}
=\sum_{\la\in\PP} T^{\abs{\la}} z^{l(\la)} = \frac{1}{(zT;T)_{\infty}},
\]
where the second equality follows from the substitution
$\la\mapsto\la'$ and the third equality follows from \eqref{Eq_GFpart}.
Hence
\[
f_r(T)=\frac{1}{(zT;T)_{\infty}},
\]
independent of $r$. Substituting this into 
\eqref{Eq_output} and \eqref{Eq_output-mod} yields
\begin{subequations}\label{Eq_app1}
\begin{equation}\label{Eq_app1a}
\sum_{\substack{\la\in\PP \\[1pt] \core{\la}=\omega}} 
T^{(\abs{\la}-\abs{\omega})/r} z^{\abs{\HHb{r}(\la)}}
=\frac{1}{(zT;T)_{\infty}(T;T)_{\infty}^{r-1}}
\end{equation}
and
\begin{equation}\label{Eq_app1b}
\sum_{\la\in\PP}T^{\abs{\la}} S^{\abs{\HH_r(\la)}} 
z^{\abs{\HHb{r}(\la)}}
=\frac{(T^r;T^r)_{\infty}^r}
{(T;T)_{\infty}(zST^r;ST^r)_{\infty}(ST^r;ST^r)_{\infty}^{r-1}}.
\end{equation}
\end{subequations}
In Section~\ref{Sec_BFN} we discuss how these results relate to a 
combinatorial identity of Buryak, Feigin and Nakajima which arose in
their work on the quasihomogeneous Hilbert scheme of points in the plane.

In order to express \eqref{Eq_app1a} and \eqref{Eq_app1b} in terms of
the modular length function \eqref{Eq_modular-length}, we replace 
$\la\mapsto\la'$ in both formulas.
In \eqref{Eq_app1a} we further make the substitution $\omega\mapsto\omega'$,
noting that $\core{\la'}=\omega'$ is equivalent to $\core{\la}=\omega$, and 
in \eqref{Eq_app1b} we use that $\abs{\HH_r(\la')}=\abs{\HH_r(\la)}$. 
This leads to the following pair of partition identities, generalising
\eqref{Eq_GFpart}.

\begin{proposition}
For $r$ a positive integer and $\omega$ an $r$-core,
\[
\sum_{\substack{\la\in\PP \\[1pt] \core{\la}=\omega}} 
T^{(\abs{\la}-\abs{\omega})/r} z^{\ell_r(\la)}
=\frac{1}{(zT;T)_{\infty}(T;T)_{\infty}^{r-1}}
\]
and
\[
\sum_{\la\in\PP}T^{\abs{\la}} S^{\abs{\HH_r(\la)}} z^{\ell_r(\la)} =
\frac{(T^r;T^r)_{\infty}^r}
{(T;T)_{\infty}(zST^r;ST^r)_{\infty}(ST^r;ST^r)_{\infty}^{r-1}}.
\]
\end{proposition}

\medskip

To motivate our second application, we recall that the exponential
generating function for the number of involutions $e_2(n)$ in the 
symmetric group $\mathfrak{S}_n$ is given by 
\cite[Equation (5.32)]{Stanley99}
\begin{equation}\label{Eq_involutions-a}
\sum_{n\geq 0} e_2(n) \frac{T^n}{n!}=\exp\bigg(T+\frac{T^2}{2}\bigg).
\end{equation}
As a direct consequence of the Robinson--Schensted correspondence, 
$e_2(n)$ is equal to the number of 
standard Young tableaux of size $n$. 
By \eqref{Eq_FRT} it thus follows that \eqref{Eq_involutions-a} 
can be written as
\begin{equation}\label{Eq_involutions-b}
\sum_{\la\in\PP} T^{\abs{\la}} 
\prod_{h\in\HH(\la)} \frac{1}{h}=\exp\bigg(T+\frac{T^2}{2}\bigg),
\end{equation}
which is to be compared with \eqref{Eq_hook}.

Somewhat surprisingly, \eqref{Eq_involutions-b} has an analogue 
for $\HHbeen(\la)$.

\begin{lemma}
We have
\begin{equation}\label{Eq_new}
\sum_{\la\in\PP} T^{\abs{\la}} 
\prod_{h\in\HHbeen(\la)} \frac{z}{h}=\exp\bigg(\frac{zT}{1-T}\bigg).
\end{equation}
\end{lemma}

\begin{proof}
Let $f(z,T)$ denote the left-hand side of \eqref{Eq_new}.
Replacing $\la$ by $\la'$, we obtain
\[
f(z,T)=\sum_{\la\in\PP} T^{\abs{\la}} \prod_{h\in\HHbeen(\la')} \frac{z}{h}
=\sum_{\la\in\PP} z^{\ell(\la)} T^{\abs{\la}}
\prod_{h\in\HHreen(\la)} \frac{1}{h}.
\]
Since $\abs{\la}=\sum_{i\geq 1} i m_i(\la)$, 
$\ell(\la)=\sum_{i\geq 1} m_i(\la)$ and
$\prod_{h\in\HHreen(\la)}h=\prod_{i\geq 1}m_i(\la)!$,
it follows that
\begin{align*}
f(z,T)=
\sum_{\la\in\PP} \prod_{i\geq 1} 
\frac{z^{m_i(\la)} T^{i m_i(\la)}}{m_i(\la)!}
&=\prod_{i\geq 1} \sum_{m_i\geq 0} \frac{(zT^i)^{m_i}}{m_i!} \\
&=\prod_{i\geq 1} \exp(zT^i)
=\exp\bigg(\frac{zT}{1-T}\bigg). \qedhere
\end{align*}
\end{proof}

After the substitution $z\mapsto z/r$ the identity \eqref{Eq_new} 
takes the form \eqref{Eq_input} with
\[
\rho(h)=\frac{z}{h}\quad\text{and}\quad f_r(T)=
\exp\bigg(\frac{z T}{r(1-T)}\bigg).
\]
By \eqref{Eq_output} and \eqref{Eq_output-mod}
we thus obtain the following modular analogues of
\eqref{Eq_new}.

\begin{proposition}
For $r$ a positive integer and $\omega$ an $r$-core,
\[
\sum_{\substack{\la\in\PP \\[1pt] \core{\la}=\omega}} 
T^{(\abs{\la}-\abs{\omega})/r} z^{\ell_r(\la)}
\prod_{h\in\HHr{r}(\la)} \frac{1}{h} = 
\frac{1}{(T;T)_{\infty}^{r-1}}\,
\exp\bigg(\frac{zT}{r(1-T)}\bigg)
\]
and
\[
\sum_{\la\in\PP}T^{\abs{\la}} S^{\abs{\HH_r(\la)}} z^{\ell_r(\la)}
\prod_{h\in\HHr{r}(\la)} \frac{1}{h}
=\frac{(T^r;T^r)_{\infty}^r}{(T;T)_{\infty}{(ST^r;ST^r)_{\infty}^{r-1}}}\,
\exp\bigg(\frac{zST^r}{r(1-ST^r)}\bigg).
\]
\end{proposition}

It is not difficult to see that by the $q$-binomial theorem 
\cite[Equation (II.3)]{GR04}
\[
\sum_{m\geq 0} \frac{(a;q)_m}{(q;q)_m}\, z^m = 
\frac{(az;q)_{\infty}}{(z;q)_{\infty}}
\]
with $(a,z,m)\mapsto (uq,zT^i,m_i)$ the identity 
\eqref{Eq_new} admits the $q$-analogue\footnote{To recover \eqref{Eq_new},
set $u=0$, replace $z\mapsto z(1-q)$, and then let $q$ tend to $1$.}
\begin{equation}\label{Eq_Amq}
\sum_{\la\in\PP} T^{\abs{\la}} 
\prod_{h\in\HHbeen(\la)} \frac{z(1-uq^h)}{1-q^h}
=\frac{(uzqT;q,T)_{\infty}}{(zT;q,T)_{\infty}}.
\end{equation}
We note that for $z=1$ this is the $t=0$ case of the $q,t$-Nekrasov--Okounkov
formula \eqref{Eq_qtNO}. Moreover, by the substitution
$(z,u)\mapsto (t,q^{-\zeta})$ followed by the limit $q\to 1$,
it simplifies to
\[
\sum_{\la\in\PP} T^{\abs{\la}} t^{\abs{\HHbeen(\la)}}
\prod_{h\in\HHbeen(\la)}
\Big(1-\frac{z}{h}\Big)=(tT;T)_{\infty}^{z-1}.
\]
For $t=1$ this was conjectured in \cite[Conjecture 2.1]{Amdeberhan12} 
and proved in \cite{HN18}.
The identity \eqref{Eq_Amq} implies the most general pair $(\rho,f_r)$ 
that we have been able to find for which $f_r$ admits a simple
closed form:
\[
\rho(h)=\frac{z(1-uq^h)}{(1-q^h)},\qquad
f_r(T)=\frac{(uzq^rT;q^r,T)_{\infty}}{(zT;q^r,T)_{\infty}}.
\]
This leads to our final application, unifying the previous two
propositions.

\begin{proposition}
For $r$ a positive integer and $\omega$ an $r$-core,
\begin{equation}\label{Eq_unify}
\sum_{\substack{\la\in\PP \\[1pt] \core{\la}=\omega}} 
T^{(\abs{\la}-\abs{\omega})/r} z^{\ell_r(\la)}
\prod_{h\in\HHr{r}(\la)} \frac{1-uq^h}{1-q^h}=
\frac{(uzq^rT;q^r,T)_{\infty}}{(T;T)_{\infty}^{r-1}(zT;q^r,T)_{\infty}}
\end{equation}
and
\[
\sum_{\la\in\PP}T^{\abs{\la}} S^{\abs{\HH_r(\la)}} z^{\ell_r(\la)}
\prod_{h\in\HHr{r}(\la)} \frac{1-uq^h}{1-q^h}
=\frac{(T^r;T^r)_{\infty}^r(uzq^rST^r;q^r,ST^r)_{\infty}}
{(T;T)_{\infty}(ST^r;ST^r)_{\infty}^{r-1}(zST^r;q^r,ST^r)_{\infty}}.
\]
\end{proposition}

\section{Conjectures and open problems}\label{Sec_Conjectures}

\subsection{A modular analogue of the $q,t$-Nekrasov--Okounkov
formula}\label{Sec_qt-modular}

By setting $z=1$ in \eqref{Eq_unify} and carrying out some elementary 
manipulations and rewritings we obtain the following pair of equivalent 
identities
\[
\sum_{\substack{\la\in\PP \\[1pt] \core{\la}=\omega}} 
T^{(\abs{\la}-\abs{\omega})/r} 
\prod_{\substack{s\in\la \\[1pt] h(s)\equiv 0 \Mod{r} \\[1pt] l(s)=0 }} 
\frac{1-uq^{a(s)+1}}{1-q^{a(s)+1}}=
\frac{(uq^rT;q^r,T)_{\infty}}{(T;T)_{\infty}^r(q^rT;q^r,T)_{\infty}}
\]
and
\[
\sum_{\substack{\la\in\PP \\[1pt] \core{\la}=\omega}} 
T^{(\abs{\la}-\abs{\omega})/r} 
\prod_{\substack{s\in\la \\[1pt] h(s)\equiv 0 \Mod{r} \\[1pt] a(s)=0}} 
\frac{1-u^{-1}t^{l(s)+1}}{1-t^{l(s)+1}}=
\frac{(u^{-1}t^rT;t^r,T)_{\infty}}{(T;T)_{\infty}^r(t^rT;t^r,T)_{\infty}}.
\]
This should be compared with the modular analogue of the 
$q$-Nekrasov--Okounkov formula
\[
\sum_{\substack{\la\in\PP \\[1pt] \core{\la}=\omega}} 
T^{(\abs{\la}-\abs{\omega})/r}
\prod_{h\in\HH_r(\la)} \frac{(1-uq^h)(1-u^{-1}q^h)}{(1-q^h)^2} \\
=\bigg(\frac{(uq^rT,u^{-1}q^rT;q^r,q^r,T)_{\infty}}
{(T,q^rT;q^r,q^r,T)_{\infty}}\bigg)^r,
\]
which, in a slightly different form, is due to 
Dehaye and Han \cite[Theorem 2]{DH11} and follows
from the $r=1$ case 
(see \cite[p.~749]{INRS12} and \cite[Theorem 5]{DH11}) 
combined with Theorem~\ref{Thm_HanJi-mod}.

The above three identities suggest a modular analogue of the 
full $q,t$-Nekrasov--Okounkov formula \eqref{Eq_qtNO}.

\begin{conjecture}
For $r$ a positive integer and $\omega$ an $r$-core,
\begin{align*}
\sum_{\substack{\la\in\PP \\[1pt] \core{\la}=\omega}} 
T^{(\abs{\la}-\abs{\omega})/r} &
\prod_{\substack{s\in\la \\[1pt] h(s)\equiv 0 \Mod{r}}}
\frac{(1-uq^{a(s)+1}t^{l(s)})(1-u^{-1}q^{a(s)}t^{l(s)+1})}
{(1-q^{a(s)+1}t^{l(s)})(1-q^{a(s)}t^{l(s)+1})} \\[1mm]
&=\frac{1}{(T;T)_{\infty}^r} 
\prod_{\substack{i,j\geq 1 \\ i+j\equiv 1 \Mod{r}}}
\frac{(uq^it^{j-1}T,u^{-1}q^{i-1}t^jT;T)_{\infty}}
{(q^it^{j-1}T,q^{i-1}t^jT;T)_{\infty}} \notag \\[1mm]
&=\frac{1}{(T;T)_{\infty}^r} 
\prod_{i=1}^r 
\frac{(uq^it^{r-i}T,u^{-1}q^{r-i}t^{i}T;q^r,t^r,T)_{\infty}}
{(q^it^{r-i}T,q^{r-i}t^iT;q^r,t^r,T)_{\infty}}.
\end{align*}
\end{conjecture}

We note that it is not at all clear why
\begin{equation}\label{Eq_rationalsum}
\sum_{\substack{\la\in\PP \\[1pt] \core{\la}=\omega}} 
T^{(\abs{\la}-\abs{\omega})/r}
\prod_{\substack{s\in\la \\[1pt] h(s)\equiv 0 \Mod{r}}}
\frac{(1-uq^{a(s)+1}t^{l(s)})(1-u^{-1}q^{a(s)}t^{l(s)+1})}
{(1-q^{a(s)+1}t^{l(s)})(1-q^{a(s)}t^{l(s)+1})}
\end{equation}
is independent of the choice of $\omega$, so that there are really 
two parts to the conjecture. Assuming this independence, we obtain
a weaker form, stated as Conjecture~\ref{Conj_modqtNO} 
in the introduction, by replacing $T\mapsto ST^r$, multiplying both sides by 
$T^{\abs{\omega}}$ and then summing $\omega$ over $\CC_r$.
It appears that neither conjecture is tractable by a multiplication-type 
theorem. In particular, non-trivial rational function identities
are behind the $\omega$-independence of \eqref{Eq_rationalsum}.
An example of such an identity in the more general elliptic case is
discussed below.

\medskip

The $q,t$-Nekrasov--Okounkov formula admits an elliptic analogue as follows.
Let $\theta(z;p)$ be the modified theta function
\[
\theta(z;p):=\sum_{n\in\mathbb{Z}} z^n q^{\binom{n}{2}}, \quad z\neq 0,
\]
and define the set of integers 
$\{C(m,\ell,n_1,n_2)\}_{m\in\mathbb{N},\,\ell,n_1,n_2\in\mathbb{Z}}$ by
\[
\frac{(upq,u^{-1}pq^{-1},upt^{-1},u^{-1}pt;p)_{\infty}}
{(pq,pq^{-1},pt^{-1},pt;p)_{\infty}} \\
=1+\sum_{m\geq 1} \, \sum_{\ell,n_1,n_2\in\Z} C(m,\ell,n_1,n_2)
p^m u^{\ell} q^{-n_1} t^{n_2}.
\]
Note that $C(m,\ell,n_1,n_2)=C(m,\ell,n_2,n_1)=C(m,-\ell,-n_1,-n_2)$
and (set $u=1$)
\begin{equation}\label{Eq_Cnul}
\sum_{\ell\in\mathbb{Z}} C(m,\ell,n_1,n_2)=0.
\end{equation}
This latter sum is well-defined since
$C(m,\ell,n_1,n_2)=0$ if $\abs{\ell}\geq\lfloor\sqrt{4m+1}\rfloor$.
Then
\begin{align}\label{Eq_eNO}
&\sum_{\la} T^{\abs{\la}} \prod_{s\in\la}
\frac{\theta(uq^{a(s)+1}t^{l(s)};p)\theta(u^{-1}q^{a(s)}t^{l(s)+1};p)}
{\theta(q^{a(s)+1}t^{l(s)};p)\theta(q^{a(s)}t^{l(s)+1};p)} \\
&\quad=
\frac{(uqT,u^{-1}tT;q,t,T)_{\infty}}{(T,qtT;q,t,T)_{\infty}} \notag \\
& \qquad \times
\prod_{m,k\geq 1} \, \prod_{\ell,n_1,n_2\in\Z} 
\bigg(\frac{(p^m T^k u^{\ell+1} q^{1-n_1} t^{n_2},
p^m T^k u^{\ell-1} q^{-n_1} t^{1+n_2};q,t)_{\infty}}
{(p^m T^k u^{\ell} q^{-n_1} t^{n_2},
p^m T^k u^{\ell} q^{1-n_1} t^{1+n_2};q,t)_{\infty}}\bigg)^
{C(km,\ell,n_1,n_2)}, \notag
\end{align}
see \cite{LLZ06,RW18,Waelder08}.
For $p=1$ this simplifies to \eqref{Eq_qtNO} and, by \eqref{Eq_Cnul},
for $u=1$ it simplifies to \eqref{Eq_GFpart} with $z=1$.

\begin{conjecture}\label{Conj_elliptic}
Let $r$ be a positive integer, $\omega$ an $r$-core and
\[
f_{\omega;r}(u;q,t,T;p) :=
\sum_{\substack{\la\in\PP \\[1pt] \core{\la}=\omega}} 
T^{(\abs{\la}-\abs{\omega})/r} 
\prod_{\substack{s\in\la \\[1pt] h(s)\equiv 0 \Mod{r}}}
\frac{\theta(uq^{a(s)+1}t^{l(s)};p)\theta(u^{-1}q^{a(s)}t^{l(s)+1};p)}
{\theta(q^{a(s)+1}t^{l(s)};p)\theta(q^{a(s)}t^{l(s)+1};p)}.
\]
Then $f_{\omega;r}(u;q,t,T;p)$ is independent of $\omega$.
\end{conjecture}

We have not yet found a (conjectural) closed-form expression for the above 
sum, except when $t=q$, see Proposition~\ref{Prop_eqNO} below.

Conjecture~\ref{Conj_elliptic} may also be stated as the claim that for all
nonnegative integers $n$ the sum
\[
f_{\omega;r,n}(u;q,t;p) :=
\sum_{\substack{\la\vdash \abs{\omega}+rn \\[1pt] \core{\la}=\omega}} \;
\prod_{\substack{s\in\la \\[1pt] h(s)\equiv 0 \Mod{r}}}
\frac{\theta(uq^{a(s)+1}t^{l(s)};p)\theta(u^{-1}q^{a(s)}t^{l(s)+1};p)}
{\theta(q^{a(s)+1}t^{l(s)};p)\theta(q^{a(s)}t^{l(s)+1};p)},
\]
which satisfies the quasi-periodicity
\[
f_{\omega;r,n}(pu;q,t;p)=\Big(\frac{t}{u^2pq}\Big)^n f_{\omega;r,n}(u;q,t;p),
\]
does not depend on the choice of $\omega\in\CC_r$. 
For $n=0$ this is 
trivially true: $f_{\omega;r,0}(u;q,t;p)=1$ for all $\omega\in\CC_r$. 
For $n=1$ it is identically true since 
\[
f_{\omega;r,1}(u;q,t;p)=
\sum_{k=1}^r
\frac{\theta(uq^kt^{r-k};p)\theta(u^{-1}q^{k-1}t^{r-k+1};p)}
{\theta(q^nt^{r-k};p)\theta(q^{k-1}t^{r-k+1};p)}
\]
independent of the $r$-core $\omega$.
For $n\geq 2$, however, theta-function addition formulas come into play.
For example, from the Littlewood decomposition it follows that
there are $2r+\binom{r}{2}$ partitions $\la$ contributing to 
$f_{\omega;r,2}(u;q,t;p)$ since the $r$-quotient $\boldsymbol{\nu}$ of 
$\la$ is given by one of the following:
\begin{itemize}
\item[(i)]
$\boldsymbol{\nu}=(\underbrace{0,\dots,0}_{i-1 \text{ times}},(2),
\underbrace{0,\dots,0}_{r-i \text{ times}})$ for
$1\leq i\leq r$, contributing $r$ terms; 

\smallskip

\item[(ii)]
$\boldsymbol{\nu}=(\underbrace{0,\dots,0}_{i-1 \text{ times}},(1^2),
\underbrace{0,\dots,0}_{r-i \text{ times}})$ for
$1\leq i\leq r$, contributing $r$ terms;

\smallskip

\item[(iii)]
$\boldsymbol{\nu}=(\underbrace{0,\dots,0}_{i-1 \text{ times}},(1),
\underbrace{0,\dots,0}_{j-i-1 \text{ times}},(1),
\underbrace{0,\dots,0}_{r-j \text{ times}})$ for 
$1\leq i<j\leq r$, contributing $\binom{r}{2}$ terms.
\end{itemize}

Accordingly, we symbolically write
\[
f_{\omega;r,2}(u;q,t;p)=
\sum_{i=1}^r \Big( \varphi^{\textrm{(i)}}_i(\omega)+
\varphi^{\textrm{(ii)}}_i(\omega)\Big)+
\sum_{1\leq i<j\leq r} \varphi^{\textrm{(iii)}}_{i,j}(\omega).
\]
Because the terms on the right sensitively depend on the
choice of $\omega\in\CC_r$, we further restrict ourselves
to a comparison of $\omega=0$ and $\omega=(1)$
(so that we require $r\geq 2$).
Then $\binom{r}{2}+2r-3$ of the terms contributing to
$f_{0;r,2}(u;q,t;p)$ have a counterpart in $f_{(1);r,2}(u;q,t;p)$
and $3$ terms in each of the sums are different.
More precisely,
\begin{align*}
\varphi^{\textrm{(i)}}_i(0)&=\begin{cases}
\varphi^{\textrm{(i)}}_i(1) & \text{for $2\leq i\leq r-1$}, \\[2mm]
\varphi^{\textrm{(i)}}_1(1) & \text{for $i=r$}, \end{cases} \\[2mm]
\varphi^{\textrm{(ii)}}_i(0)&=\begin{cases}
\varphi^{\textrm{(ii)}}_i(1) & \text{for $2\leq i\leq r-1$}, \\[2mm]
\varphi^{\textrm{(ii)}}_r(1) & \text{for $i=1$}, \end{cases}
\intertext{and}
\varphi^{\textrm{(iii)}}_{i,j}(0)&=\begin{cases}
\varphi^{\textrm{(iii)}}_{i,j}(1) & 
\text{for $1<i<j<r$}, \\[2mm]
\varphi^{\textrm{(iii)}}_{j,r}(1) &
\text{for $i=1$, $1<j<r$}, \\[2mm]
\varphi^{\textrm{(iii)}}_{1,i}(1) &
\text{for $j=r$, $1<i<r$}.
\end{cases}
\end{align*}
Hence
\begin{multline*}
f_{0;r,2}(u;q,t;p)-f_{(1);r,2}(u;q,t;p)
=\varphi^{\textrm{(i)}}_1(0)
+\varphi^{\textrm{(ii)}}_r(0)
+\varphi^{\textrm{(iii)}}_{1,r}(0) \\
-\varphi^{\textrm{(i)}}_r(1)
-\varphi^{\textrm{(ii)}}_1(1)
-\varphi^{\textrm{(iii)}}_{1,r}(1).
\end{multline*}
If by slight abuse of notation we index the above by the actual partitions
$\la$ this yields
\begin{multline*}
f_{0;r,2}(u;q,t;p)-f_{(1);r,2}(u;q,t;p)
=\varphi_{(r+1,1^{r-1})}+\varphi_{(r,1^r)}+\varphi_{(r,2,1^{r-2})} \\
-\varphi_{(r,2,1^{r-1})}-\varphi_{(r+1,2,1^{r-2})}-\varphi_{(r+1,1^r)}.
\end{multline*}
To see that this vanishes we note that
\begin{align*}
&\varphi_{(r+1,1^{r-1})}-\varphi_{(r+1,1^r)} \\
&\qquad=\frac{\theta(uq^r;p)\theta(uq^{r+1}t^{r-1};p)
\theta(u^{-1}q^{r-1}t;p)\theta(u^{-1}q^rt^r;p)}
{\theta(q^r;p)\theta(q^{r+1}t^{r-1};p)
\theta(q^{r-1}t;p)\theta(q^rt^r;p)} \\
&\qquad\quad-\frac{\theta(uq^r;p)\theta(uqt^{r-1};p)
\theta(u^{-1}q^{r-1}t;p)\theta(u^{-1}t^r;p)}
{\theta(q^r;p)\theta(qt^{r-1};p)
\theta(q^{r-1}t;p)\theta(t^r;p)} \\
&\qquad=-q^{2r-1}t^{r+1}\,\frac{\theta(uqt^{-1};p)\theta(uq^{1-r}t^{-1};p)
\theta(u^{-1}q^{-r};p)\theta(u^{-1};p)\theta(q^{r+1}t^{2r-1};p)}
{\theta(q^{r-1}t;p)\theta(qt^{r-1};p)\theta(t^r;p)
\theta(q^{r+1}t^{r-1};p)\theta(q^rt^r;p)} =: t_1.
\end{align*}
Here the last equality follows from the addition formula 
\cite[p.~451, Example 5]{WW96}
\begin{multline}\label{Eq_addition}
\theta(xz;p)\theta(x/z;p)\theta(yw;p)\theta(y/w;p)-
\theta(xw;p)\theta(x/w;p)\theta(yz;p)\theta(y/z;p) \\
=\frac{y}{z}\, \theta(xy;p)\theta(x/y;p)\theta(zw;p)\theta(z/w;p)
\end{multline}
as well as $\theta(z;p)=-z\theta(1/z;p)$.
Similarly, we have
\begin{align*}
&\varphi_{(r,1^r)}-\varphi_{(r,2,1^{r-1})} \\
&\qquad=\frac{\theta(uqt^{r-1};p)\theta(uq^rt^r;p)
\theta(u^{-1}t^r;p)\theta(u^{-1}q^{r-1}t^{r+1};p)}
{\theta(qt^{r-1};p)\theta(q^rt^r;p)\theta(t^r;p)\theta(q^{r-1}t^{r+1};p)} \\
&\qquad\quad-\frac{\theta(uq^{r-1}t;p)\theta(uq^rt^r;p)
\theta(u^{-1}q^{r-2}t^2;p)\theta(u^{-1}q^{r-1}t^{r+1};p)}
{\theta(q^{r-1}t;p)\theta(q^rt^r;p)\theta(q^{r-2}t^2;p)
\theta(q^{r-1}t^{r+1};p)} \\
&\qquad=\frac{\theta(uqt^{-1};p)\theta(uq^rt^r;p)
\theta(u^{-1}q^{r-1}t^{r+1};p)
\theta(u^{-1};p)\theta(q^{2-r}t^{r-2};p)}
{\theta(q^{r-1}t;p)\theta(q^{2-r}t^{-2};p)\theta(qt^{r-1};p) 
\theta(t^r;p)\theta(q^rt^r;p)}=:t_2
\end{align*}
and
\begin{align*}
&\varphi_{(r,2,1^{r-2})}-\varphi_{(r+1,2,1^{r-2})} \\
&\qquad=\frac{\theta(uq^{r-1}t;p)\theta(uq^2t^{r-2};p)
\theta(u^{-1}q^{r-2}t^2;p)\theta(u^{-1}qt^{r-1};p)}
{\theta(q^{r-1}t;p)\theta(q^2t^{r-2};p)\theta(q^{r-2}t^2;p)
\theta(qt^{r-1};p)} \\ 
&\qquad\quad-
\frac{\theta(uq^2t^{r-2};p)\theta(uq^{r+1}t^{r-1};p)
\theta(u^{-1}qt^{r-1};p)\theta(u^{-1}q^rt^r;p)}
{\theta(q^2t^{r-2};p)\theta(q^{r+1}t^{r-1};p)\theta(qt^{r-1};p)
\theta(q^rt^r;p)} \\
&\qquad=
-\frac{\theta(uq^2t^{r-2};p)\theta(uqt^{-1};p)\theta(u^{-1}qt^{r-1};p)
\theta(u^{-1};p)\theta(q^{2r-1}t^{r+1};p)}
{\theta(q^{r-1}t;p)\theta(q^{2-r}t^{-2};p)\theta(qt^{r-1};p)
\theta(q^{r+1}t^{r-1};p)\theta(q^rt^r;p)}=:t_3.
\end{align*}
By one more application of \eqref{Eq_addition} it follows that
$t_1+t_2+t_3=0$, and hence that
\[
f_{0;r,2}(u;q,t;p)=f_{(1);r,2}(u;q,t;p)
\]
for $r\geq 2$.

If we set $t=q$ in the elliptic Nekrasov--Okounkov formula \eqref{Eq_eNO}
we obtain
\begin{multline*}
\sum_{\la} T^{\abs{\la}} \prod_{h\in\HH(\la)}
\frac{\theta(uq^h;p)\theta(u^{-1}q^h;p)}{\theta(q^h;p)\theta(q^h;p)} 
=
\frac{(uqT,u^{-1}qT;q,q,T)_{\infty}}{(T,q^2T;q,q,T)_{\infty}}  \\ \times
\prod_{m,k\geq 1} \, \prod_{\ell,n_1,n_2\in\Z} 
\bigg(\frac{(p^m T^k u^{\ell+1} q^{n_2-n_1+1},
p^m T^k u^{\ell-1} q^{n_2-n_1+1};q,q)_{\infty}}
{(p^m T^k u^{\ell} q^{n_2-n_1},
p^m T^k u^{\ell} q^{n_2-n_1+2};q,q)_{\infty}}
\bigg)^{C(km,\ell,n_1,n_2)}.
\end{multline*}
By Theorem~\ref{Thm_HanJi} this implies our next result.

\begin{proposition}[A $p,q$-Nekrasov--Okounkov formula]\label{Prop_eqNO}
For $r$ a positive integer,
\begin{align*}
&\sum_{\la\in\PP}T^{\abs{\la}}S^{\abs{\HH_r(\la)}}
\prod_{h\in \HH_r(\la)} 
\frac{\theta(uq^h;p)\theta(u^{-1}q^h;p)}{\theta(q^h;p)\theta(q^h;p)} \\[2mm]
&=\frac{(T^r;T^r)_{\infty}^r}{(T;T)_{\infty}}\,
\bigg(
\frac{(uq^rST^r,u^{-1}q^rST^r;q^r,q^r,ST^r)_{\infty}}
{(ST^r,q^{2r}ST^r;q^r,q^r,ST^r)_{\infty}}\bigg)^r \\[2mm]
&\;\, \times
\prod_{m,k\geq 1} \, \prod_{\ell,n_1,n_2\in\Z} 
\bigg(\frac{(p^m S^kT^{kr} u^{\ell+1} q^{(n_2-n_1+1)r},
p^m S^kT^{kr} u^{\ell-1} q^{(n_2-n_1+1)r};q^r,q^r)_{\infty}}
{(p^m S^kT^{kr} u^{\ell} q^{(n_2-n_1)r},
p^m S^kT^{kr} u^{\ell} q^{(n_2-n_1+2)r};q^r,q^r)_{\infty}}
\bigg)^{C(km,\ell,n_1,n_2)r}.
\end{align*}
\end{proposition}

This naturally leads to the following open problem.

\begin{problem}
For $r\geq 2$, identify the integers $B_r(m,\ell,n_1,n_2)$ in the expansion
\begin{align*}
\sum_{\la\in\PP}&T^{\abs{\la}}S^{\abs{\HH_r(\la)}}
\prod_{\substack{s\in\la \\[1pt] h(s)\equiv 0 \Mod{r}}}
\frac{\theta(uq^{a(s)+1}t^{l(s)},u^{-1}q^{a(s)}t^{l(s)+1};p)}
{\theta(q^{a(s)+1}t^{l(s)},q^{a(s)}t^{l(s)+1};p)} \\[2mm]
&=\frac{(T^r;T^r)_{\infty}^r}{(T;T)_{\infty}(ST^r;ST^r)_{\infty}^r}
\prod_{i=1}^r 
\frac{(uq^it^{r-i}ST^r,u^{-1}q^{r-i}t^iST^r;q^r,t^r,ST^r)_{\infty}}
{(q^it^{r-i}ST^r,q^{r-i}t^iST^r;q^r,t^r,ST^r)_{\infty}} \\[2mm]
&\quad\qquad \times
\prod_{m,k\geq 1} \prod_{\ell,n_1,n_2\in\Z}
\big(1-p^m (ST^r)^k u^{\ell} q^{n_1} t^{n_2}\big)^{B_r(mk,\ell,n_1,n_2)}.
\end{align*}
\end{problem}

\subsection{The Buryak--Feigin--Nakajima formula}\label{Sec_BFN}

Let $(\mathbb{C}^2)^{[n]}$ be the Hilbert scheme of points in the
plane, parametrising the ideals $I$ of $\mathbb{C}[x,y]$ of colength
$n$.
The action of the torus $(\mathbb{C}^{\ast})^2$ on 
$\mathbb{C}^2$ given by $(s,t)\cdot(x,y)=(sx,ty)$ lifts to
an action on $(\mathbb{C}^2)^{[n]}$.
For $\alpha,\beta$ nonnegative integers such that $\alpha+\beta\geq 1$,
define the one-dimensional subtorus $T_{\alpha,\beta}$ of 
$(\Complex^{\ast})^2$ by
$T_{\alpha,\beta}:=\{(t^{\alpha},t^{\beta}):~t\in\Complex^{\ast}\}$.
When both $\alpha$ and $\beta$
are strictly positive, the set of fixed points
\[
(\Complex^2)^{[n]}_{\alpha,\beta}:=
\big((\Complex^2)^{[n]}\big)^{T_{\alpha,\beta}}
\]
parametrises quasi-homogeneous ideals $I\subseteq \mathbb{C}[x,y]$
of colength $n$, and is known as the quasihomogeneous Hilbert 
scheme \cite{BF13,Evain04}.
Let $P(X;z):=\sum_{i\geq 0} \dim H_i(X;\mathbb{Q}) z^{i/2}$ denote
the Poincar\'e polynomial of a manifold $X$, where $H_i$ is the $i$th
homology group (over $\mathbb{Q}$) and $\dim H_i$ the $i$th Betti number. 
Proving an earlier conjecture of Buryak \cite[Conjecture~1.4]{Buryak12},  
Buryak and Feigin \cite[Theorem~1]{BF13} proved the following 
beautiful identity for the generating function  of the Poincar\'e 
polynomial of the quasihomogeneous Hilbert scheme.
\begin{theorem}
For $\alpha,\beta$ positive integers such that $\gcd(\alpha,\beta)=1$,
let $r:=\alpha+\beta$. Then
\[
\sum_{n=0}^{\infty} P\big((\Complex^2)^{[n]}_{\alpha,\beta};z\big) T^n=
\frac{(T^r;T^r)_{\infty}}{(T;T)_{\infty}(zT^r;T^r)_{\infty}}.
\]
\end{theorem}
Subsequently, Buryak, Feigin and Nakajima \cite{BFN15} obtained a
more general result which eliminates the need for the restriction
that $\alpha$ and $\beta$ are coprime. To this end the torus
$T_{\alpha,\beta}$ is replaced by 
$T_{\alpha,\beta}\times\Gamma_{\alpha+\beta}$,
where 
\[
\Gamma_r:=\big\{(\eup^{2\pi\iup k/r},\eup^{-2\pi\iup k/r})
\in (\mathbb{C}^{\ast})^2: 0\leq k\leq r-1\big\},
\]
and singular homology is replaced by Borel--Moore (BM) homology.

\begin{theorem}
For $\alpha,\beta$ nonnegative integers such that $\alpha+\beta\geq 1$,
let $r:=\alpha+\beta$. Then
\[
\sum_{n=0}^{\infty} P^{\textup{BM}}\Big(
\big((\Complex^2)^{[n]}_{\alpha,\beta}\big)^{\Gamma_{\alpha+\beta}};
z\Big) T^n=
\frac{(T^r;T^r)_{\infty}}{(T;T)_{\infty}(zT^r;T^r)_{\infty}},
\]
where
\[
P^{\textup{BM}}(X;z)
:=\sum_{i\geq 0} \dim H^{\textup{BM}}_i(X;\mathbb{Q}) z^{i/2}.
\]
\end{theorem}
For positive, coprime $\alpha$ and $\beta$,
$\Gamma_{\alpha+\beta}\subseteq T_{\alpha+\beta}$, so that
$((\Complex^2)^{[n]}_{\alpha,\beta})^{\Gamma_{\alpha+\beta}}
=(\Complex^2)^{[n]}_{\alpha,\beta}$.
Positivity also implies compactness, in which case both homology theories 
are equivalent.
Hence the second theorem contains the first as special case.

Both theorems admit a purely combinatorial description in terms of a 
statistic on partitions introduced by Buryak and Feigin in \cite{BF13}
(and refined in \cite{BFN15} to $\gcd(\alpha,\beta)\neq 1$).
For $\alpha\geq 1$ and $\beta\geq 0$ a pair of integers,
define $\BFset_{\alpha,\beta}(\la)\subseteq\la$ and
$\BF_{\alpha,\beta}(\la)\in \mathbb{N}_0$ by
\[
\BFset_{\alpha,\beta}(\la)=
\big\{s\in\la:~\alpha\, l(s)=\beta\, a(s)+\beta \text{ and }
h(s)\equiv 0 \pmod{\alpha+\beta} \big\}
\]
and
\[
\BF_{\alpha,\beta}(\la):=\abs{\BFset_{\alpha,\beta}(\la)}.
\]

For example, for the partition $\la=(7,6,4,4,2,1)$ the sets
$\BFset_{4,2}(\la)\subset\BFset_{2,1}(\la)$ 
are given by coloured squares in the two diagrams below:

\medskip

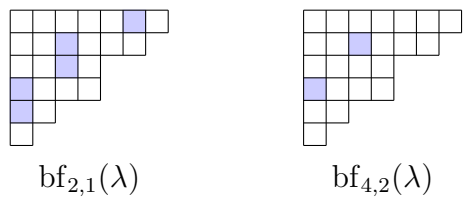
\begin{center}
\begin{tikzpicture}[scale=0.3]
\tikzmath{\x=13;}
\tikzmath{\y=26;}

\filldraw[blue!20] (5,0) rectangle (6,-1);
\filldraw[blue!20] (2,-1) rectangle (3,-3);
\filldraw[blue!20] (0,-3) rectangle (1,-5);
\draw (0,0)--(7,0);
\draw (0,-1)--(7,-1);
\draw (0,-2)--(6,-2);
\draw (0,-3)--(4,-3);
\draw (0,-4)--(4,-4);
\draw (0,-5)--(2,-5);
\draw (0,-6)--(1,-6);
\draw (0,0)--(0,-6);
\draw (1,0)--(1,-6);
\draw (2,0)--(2,-5);
\draw (3,0)--(3,-4);
\draw (4,0)--(4,-4);
\draw (5,0)--(5,-2);
\draw (6,0)--(6,-2);
\draw (7,0)--(7,-1);
\draw (3.5,-7.5) node {$\BFset_{2,1}(\la)$};
\filldraw[blue!20] (\x+2,-1) rectangle (\x+3,-2);
\filldraw[blue!20] (\x+0,-3) rectangle (\x+1,-4);
\draw (\x+0,0)--(\x+7,0);
\draw (\x+0,-1)--(\x+7,-1);
\draw (\x+0,-2)--(\x+6,-2);
\draw (\x+0,-3)--(\x+4,-3);
\draw (\x+0,-4)--(\x+4,-4);
\draw (\x+0,-5)--(\x+2,-5);
\draw (\x+0,-6)--(\x+1,-6);
\draw (\x+0,0)--(\x+0,-6);
\draw (\x+1,0)--(\x+1,-6);
\draw (\x+2,0)--(\x+2,-5);
\draw (\x+3,0)--(\x+3,-4);
\draw (\x+4,0)--(\x+4,-4);
\draw (\x+5,0)--(\x+5,-2);
\draw (\x+6,0)--(\x+6,-2);
\draw (\x+7,0)--(\x+7,-1);
\draw (\x+3.5,-7.5) node {$\BFset_{4,2}(\la)$};
\end{tikzpicture}
\end{center}
Hence $\BF_{2,1}(\la)=5$ and $\BF_{4,2}(\la)=2$.
Obviously, we further have $\BF_{2k,k}(\la)=0$ for $k\geq 3$.

Note that if $\alpha$ and $\beta$ are positive coprime integers
then we may drop the congruence condition on the hook-lengths of $s$.
Indeed, $\alpha l(s)=\beta a(s)+\beta$ implies $\alpha\mid a(s)+1$ and
$\beta\mid l(s)$, so that
\[
h(s)=a(s)+l(s)+1=(\alpha/\beta+1)l(s)=(\alpha+\beta)\,\frac{l(s)}{\beta}
\equiv 0\pmod{\alpha+\beta}.
\]
This is in accordance with 
$\Gamma_{\alpha+\beta}\subset T_{\alpha,\beta}$ 
for positive coprime $\alpha$ and $\beta$.

By the Bialynicki-Birula theorem \cite{BB73,BB76},
$(\Complex^2)^{[n]}_{\alpha,\beta}$ has a cellular decomposition
with cells
\[
C_p=\big\{z\in (\Complex^2)^{[n]}_{\alpha,\beta}:~
\lim_{t\to 0} tz=p \text{ for } t\in T_{1,\gamma}\big\}
\]
where $\gamma$ is sufficiently large and $p$ is a fixed point 
of the $(\Complex^{\ast})^2$ action. If $p$ is indexed by
the partition $\la\vdash n$, then \cite{BF13}
\[
\dim C_p=\BF_{\alpha,\beta}(\la).
\]
This decomposition carries over mutatis mutandis to\footnote{In \cite{BFN15}
the case $\alpha=0$ is included in the statement, but we believe this to
be a minor slip.}
$\big((\Complex^2)^{[n]}_{\alpha,\beta}\big)^{\Gamma_{\alpha+\beta}}$.
Hence
\[
P^{\textup{BM}}\Big(
\big((\Complex^2)^{[n]}_{\alpha,\beta}\big)^{\Gamma_{\alpha+\beta}};
z\Big)=\sum_{\la\vdash n} z^{\BF_{\alpha,\beta}(\la)},
\]
resulting in the following partition theorem \cite[Theorem~2]{BF13} and
\cite[Corollary~1.3]{BFN15}.

\begin{theorem}\label{Thm_BF}
For integers $\alpha\geq 1$ and $\beta\geq 0$, let $r:=\alpha+\beta$. Then
\[
\sum_{\la\in\PP} T^{\abs{\la}} z^{\BF_{\alpha,\beta}(\la)} =
\frac{(T^r;T^r)_{\infty}}{(T;T)_{\infty}(zT^r;T^r)_{\infty}}.
\]
\end{theorem}

What makes this theorem difficult to prove purely combinatorially
is that it is not at all clear why
\[
\BF_{\alpha,\beta}(n)=\BF_{\alpha+\beta,0}(n),
\]
where $\BF_{\alpha,\beta}(n)$ is the multiset
\[
\BF_{\alpha,\beta}(n):=\{\BF_{\alpha,\beta}(\la): \la\vdash n\}.
\]
In fact, if one carefully inspects the details of the proof of 
Theorem~\ref{Thm_BF} in \cite[Section 3]{BF13} it follows that, 
more strongly,
\[
\BF_{\alpha,\beta}(n;\omega)=\BF_{\alpha+\beta,0}(n;\omega),
\]
where, for $\omega\in\CC_r$,
\[
\BF_{\alpha,\beta}(n;\omega):=\{\BF_{\alpha,\beta}(\la): \la\vdash n,
\textup{$(\alpha+\beta)$-core}(\la)=\omega\}.
\]
For example, for all
\[
(\alpha,\beta)\in\{(3,0),(2,1),(1,2)\},
\]
we have
\begin{gather*}
\BF_{\alpha,\beta}(9;0)=\{0^{10},1^8,2^3,3\}, \\
\BF_{\alpha,\beta}(9;(4,2))=\BF_{\alpha,\beta}(9;(2,2,1,1))=\{0^2,1\}, \\
\BF_{\alpha,\beta}(9;(5,3,1))=\BF_{\alpha,\beta}(9;(3,2,2,1,1))=\{0\}
\end{gather*}
and $\BF_{\alpha,\beta}(9;\omega)=\emptyset$ for all other 
$\omega\in\CC_3$, accounting for the $30$ partitions of $9$.
We thus have a slightly stronger result as follows.

\begin{proposition}
For integers $\alpha\geq 1$ and $\beta\geq 0$, let $r:=\alpha+\beta$. Then
\begin{subequations}\label{Eq_BFz}
\begin{equation}\label{Eq_BFza}
\sum_{\substack{\la\in\PP \\[1pt] \core{\la}=\omega}} 
T^{(\abs{\la}-\abs{\omega})/r} z^{\BF_{\alpha,\beta}(\la)}
=\frac{1}{(zT;T)_{\infty}(T;T)_{\infty}^{r-1}}
\end{equation}
for $\omega\in\CC_r$, and
\begin{equation}\label{Eq_BFzb}
\sum_{\la\in\PP} T^{\abs{\la}} S^{\abs{\HH_r(\la)}} 
z^{\BF_{\alpha,\beta}(\la)} 
=\frac{(T^r;T^r)_{\infty}^r}{(T;T)_{\infty}
(zST^r;ST^r)_{\infty}(ST^r;ST^r)_{\infty}^{r-1}}.
\end{equation}
\end{subequations}
\end{proposition}

Here \eqref{Eq_BFzb} follows from \eqref{Eq_BFza} by the exact same 
reasoning which shows that \eqref{Eq_output} and \eqref{Eq_output-mod} are 
equivalent, see the proof on page~\pageref{page_pf}.

Since $\BFset_{r,0}(\la)$ is exactly the set of bottom squares
of $\la$ whose hook-lengths are congruent to $0$ modulo $r$, it follows
that
\[
\BF_{r,0}(\la)=\abs{\HHb{r}(\la)}.
\]
The pair of identities \eqref{Eq_BFz} are thus a generalisation of 
\eqref{Eq_app1}.
Moreover, Theorem~\ref{Thm_new} may be restated as
the claim if
\[
f_r(T):=
\sum_{\la\in\PP} T^{\abs{\la}} 
\prod_{s\in\BFset_{1,0}(\la)} \rho\big(r h(s)\big),
\]
then
\[
\sum_{\substack{\la\in\PP \\[1pt] \core{\la}=\omega}} 
T^{(\abs{\la}-\abs{\omega})/r} 
\prod_{s\in\BFset_{r,0}(\la)} \rho\big(h(s)\big) = 
\frac{f_r(T)}{(T;T)_{\infty}^{r-1}}\quad
\]
and
\[
\sum_{\la\in\PP}T^{\abs{\la}} S^{\abs{\HH_r(\la)}}
\prod_{s\in\BFset_{r,0}(\la)} \rho\big(h(s)\big) = 
\frac{(T^r;T^r)_{\infty}^r}{(T;T)_{\infty}(ST^r;ST^r)_{\infty}^{r-1}}
\,f_r(ST^r)
\]
for $r$ a positive integer and $\omega$ an $r$-core.

\begin{problem}
Extend the above to all $\BFset_{\alpha,\beta}(\la)$.
\end{problem}

If such an extension exists, it is clear that $h(s)$
(which for $s\in\BFset_{r,0}(\la)$ is equal to $a(s)+1$) should be
replaced by a more complicated statistic on the squares of $\la$.

\end{document}